\documentclass [12pt]{article}
\usepackage{latexsym,amssymb,amsfonts,amsmath,amsthm,graphicx}
\usepackage{hyperref}
\usepackage{color}

\newtheorem{lemma}{Lemma}[section]
\newtheorem{proposition}{Proposition}[section]

\newtheorem{theorem}{Theorem}[section]
\newtheorem{remark}{Remark}[section]

\newcommand{\s}{\vspace{2ex}}
\newcommand{\n}{\noindent}
\newcommand{\e}{\varepsilon}
\newcommand{\R}{\mathbb{R}}
\newcommand{\Z}{\mathbb{Z}}

\title{Instability of a nonlinear oscillator with small friction and small additive noise}

\author{Peter H. Baxendale} 

\AtEndDocument{\bigskip{\footnotesize
  \noindent
  \textsc{Peter H. Baxendale. Department of Mathematics, University of Southern California, CA 90089, USA.}
  \textit{E-mail address}: \texttt{\href{mailto:baxendal@usc.edu}{baxendal@usc.edu}}
 }}

\begin{document}

\maketitle

\begin{abstract}

Let $\lambda = \lambda(\beta,\sigma,a,b)$ denote the top Lyapunov exponent for the linearization along trajectories of the noisy damped non-linear oscillator $\ddot{x}+\beta \dot{x} + ax+bx^3 = \sigma \dot{W}_t$, where $a$, $b$ and $\beta$ are all positive and $\sigma \neq 0$.  In 2004 Arnold, Imkeller and Sri Namachchivaya stated without proof that $\lambda(\e^2 \beta,\e \sigma,a,b) \sim \overline{\lambda} \e^{2/3}$ as $\e \to 0$ with $\overline{\lambda} > 0$.  This paper contains a proof of this assertion.   

\end{abstract}

\s

\n Keywords: non-linear oscillator, Lyapunov exponent, Hamiltonian system, Furstenberg-Khas'minskii formula.\\
2020 Mathematics Subject Classification: 60H10, 60J60, 37H15, 37A30.

\section{Introduction} \label{sec intro}

Consider the stochastic nonlinear oscillator equation
 \begin{equation} \label{nonlin}
   \ddot{x}_t + \beta \dot{x}_t + {\cal U}'(x_t) = \sigma \dot{W}_t
  \end{equation}
with $\beta > 0$ and $\sigma \neq 0$, where ${\cal U}(x)$ is the single well potential ${\cal U}(x) = ax^2/2+ bx^4/4$ with $a,b > 0$.  In phase space we have the system
\begin{equation} \label{add ab}
  \begin{split}
     dx_t &  =y_t\,dt\\
     dy_t & = (- ax_t -bx_t^3 -  \beta y_t)\,dt+ \sigma \circ dW_t.
    \end{split}
    \end{equation}    
It is well known, see for example Mattingly, Stuart and Higham \cite{MSH02}, that the process $\{(x_t,y_t): t \ge 0\}$ is positive recurrent on $\R^2$ with unique invariant probability measure $\mu$ with density $\rho(x,y)$ of the form 
   \begin{equation} \label{density}
    \rho(x,y) = C \exp\left\{\frac{-2\beta}{\sigma^2}\left(\frac{ax^2}{2}+ \frac{bx^4}{4}+ \frac{y^2}{2}\right)\right\}
              \end{equation}
where $C > 0$ is the normalizing constant.   
Linearizing \eqref{add ab} along the trajectory $\{(x_t,y_t): t \ge 0\}$ we get the derivative process $\{v_t: t \ge 0\}$ in $\R^2$ given by  
   \begin{equation} \nonumber 
      dv_t = \begin{bmatrix} 0 & 1 \\-a-3bx_t^2 & - \beta \end{bmatrix}v_t dt.
   \end{equation}   
The (top) Lyapunov exponent is the almost sure limit
    \begin{equation} \label{lam ab}
      \lambda(\beta,\sigma,a,b) = \lim_{t \to \infty} \frac{1}{t} \log \|v_t\|.
      \end{equation}
We will see in Proposition \ref{prop FK cart} that the right side of \eqref{lam ab} is an almost sure limit with the same value for all initial values $(x_0,y_0,v_0)$ with $v_0 \neq 0$.

Arnold, Imkeller and Sri Namachchivaya \cite{AIN04} study the case of small noise and small dissipation by replacing  $\sigma \rightsquigarrow \e \sigma$ and $\beta \rightsquigarrow \e^2 \beta$, and they state without proof that $\lambda(\e^2 \beta,\e \sigma,a,b) \sim \e^{2/3} \overline{\lambda}$ as $\e \to 0$ for some constant $\overline{\lambda} > 0$.  In this paper we present a proof of this assertion.

\s

Replacing $(x_t,y_t,W_t)$ with $(Bx_{At},ABy_{At},A^{-1/2}W_{At})$ gives a version of \eqref{add ab} with the parameters $(\beta,\sigma,a,b)$ replaced with parameters $(A\beta, A^{3/2}\sigma/B, A^2a, A^2B^2 b)$.  It follows that
    $$
    \lambda( \beta, \sigma, a,b) = \frac{1}{A} \lambda(A\beta, A^{3/2}\sigma/B, A^2 a/B, A^2B^2b)
    $$
for all $A, B > 0$, and in particular
      $$
    \lambda(\beta,   \sigma, a,b) = a^{1/2} \lambda( a^{-1/2}\beta,   a^{-5/4} b^{1/2} \sigma,1,1).
    $$
So without loss of generality we may assume $a=b=1$.     
With $a=b=1$ and the replacement $\beta \rightsquigarrow \e^2 \beta$ and $\sigma \rightsquigarrow \e \sigma$ the system becomes
   \begin{equation} \label{add eps}
  \begin{split}
     dx_t &  =y_t\,dt\\
     dy_t & = (- x_t -x_t^3 -  \e^2\beta y_t)\,dt+ \e \sigma \circ dW_t.
    \end{split}
    \end{equation} 

\s

The system \eqref{add eps} may be written in the form   
 \begin{equation} \label{sde}
    d(x_t,y_t) = \overline{\nabla} H(x_t,y_t)dt + \e^2V_0(x_t,y_t)dt+ \e V_1(x_t,y_t) \circ dW_t
    \end{equation}
representing the Hamiltonian flow for a function $H(x,y)$ with a small perturbation as $\e \to 0$ determined by the drift and noise vector fields $V_0(x,y)$ and $V_1(x,y)$.   The calculations in this paper start with the framework of  Baxendale and Goukasian \cite{BG02} which deals with \eqref{sde} in the case when $H(x,y)$ has a single critical point at $(0,0)$ and $H(x,y) \to \infty$ as $\|(x,y)\| \to \infty$ and $V_0(x,y)$ and $V_1(x,y)$ are general vector fields.   
 
We will use the general notation of \eqref{sde} as long as possible, before restricting to the special case \eqref{add eps} by replacing the general function $H$ and vector fields $V_0$ and $V_1$ with the specific formulas
\begin{equation} \label{H}
H(x,y) = {\cal U}(x)+ y^2/2= x^2/2+ x^4/4 + y^2/2
  \end{equation} and 
 \begin{equation} \label{V}
 V_0(x,y) = \begin{bmatrix} 0 \\ - \beta y \end{bmatrix}, \qquad V_1(x,y) = \begin{bmatrix} 0 \\ \sigma \end{bmatrix}.
 \end{equation}  
The reader is cautioned to distinguish carefully between results for the general system \eqref{sde} and results for the particular system \eqref{add eps}.

\subsection{Main result} \label{sec main} 

Define orthogonal vector fields 
  \begin{equation} \label{U}
   U_1(x,y) = \overline{\nabla} H(x,y) \quad \quad \mbox{ and } \quad \quad
       U_2(x,y) = \frac{\nabla H(x,y)}{\|\nabla
H(x,y)\|^2}
    \end{equation} for $(x,y) \neq (0,0)$.  Using the notation $V.F$ to denote the action of a vector field $V$ as a differential operator on a function $F$, we have $U_1.H(x,y) = 0$ and $U_2.H(x,y) = 1$.

For $h > 0$ let the probability measure $m_h$ denote the normalized occupation time measure for the Hamiltonian flow around the orbit $H^{-1}(h) = \{(x,y): H(x,y) = h\}$, and let $T(h)$ denote the period of the Hamiltonian flow around $H^{-1}(h)$.  Define
    \begin{align}  \label{G1}    
      G_1(h) & = -\frac{T'(h)}{T(h)}, \\
      G_2(h)& = \int_{H^{-1}(h)} \bigl((\overline{\nabla}H.(V_1.H))(x,y)\bigr)^2 dm_h(x,y), \label{G2}\\
      G(h) & = |G_1(h)|^{2/3}\bigl(G_2(h)\bigr)^{1/3}. \label{G}
      \end{align}
 In the following result the measure $\mu$ denotes the invariant probability measure for the system \eqref{add eps} in $\R^2$.  It does not depend on $\e > 0$ and has density $\rho(x,y)$ given by \eqref{density} with $a = b = 1$.

\begin{theorem}   \label{thm main} For fixed $\beta > 0$ and $\sigma \neq 0$ the Lyapunov exponent for the system \eqref{add eps} satisfies   \begin{equation} \nonumber 
  \lambda(\e^2 \beta,\e \sigma,1,1) = \e^{2/3} \overline{\lambda} + {\cal O}(\e^{4/3})
  \end{equation}
  as $\e \to 0$, where
    \begin{equation} \nonumber 
    \overline{\lambda} = \gamma_0 \int_{\R^2} G(H(x,y)) d\mu(x,y) \in (0,\infty)
    \end{equation}
and $\gamma_0 = \pi/\bigl(2^{1/3}3^{1/6}[\Gamma(1/3)]^2\bigr) \approx 0.29$.  
\end{theorem}

A related integral formula for $\overline{\lambda}$ in terms of an associated action-angle coordinate system in $\R^2 \setminus \{(0,0)\}$ is stated without proof in \cite{AIN04}.

Exact formulas for $T(h)$, $G_1(h)$ and $G_2(h)$, and hence for $G(h)$, in terms of complete elliptic integrals are given in Section \ref{sec orbitave}. In particular $G_1(h) > 0$ for all $h >0$. 
     The property $T'(h) < 0$ implies that orbits with larger $h$ have smaller periods, so we may interpret $G_1(h)$ as the time averaged shearing effect of the Hamiltonian flow in a small neighborhood of the orbit $H^{-1}(h)$.   More precisely,  consider $(x,y)$ with $H(x,y) = h$, and a nearby point $(\widetilde{x},\widetilde{y})= (x,y) + \delta U_2(x,y)$, so that $H(\widetilde{x},\widetilde{y}) = h+\delta+{\cal O}(\delta^2)$.  Using \eqref{dw0} and \eqref{int J} we see that after time $T(h)$ the point $(x,y)$ has returned to $(x,y)$ but the nearby point $(\widetilde{x},\widetilde{y})$ is now at $(\widetilde{x},\widetilde{y}) - \delta T'(h)U_1(\widetilde{x},\widetilde{y}) + {\cal O}(\delta^2)$.       

Also $V_1.H(x,y)$ corresponds to the change in $H(x,y)$ caused by a ``kick'' $V_1(x,y)$, and then $U_1.(V_1.H))$ is the variation of the effect of the kick along the orbit.  Thus $G_2(h)$ corresponds to the mean square variation along the orbit of the random kicking effect of the noise vector field $V_1$.  
   Theorem \ref{thm main} quantifies how the shearing and the random kicking effects on the Hamiltonian orbits together contribute to the positivity of the Lyapunov exponent.  Related results on the effects of shearing and periodic or random kicking on limit cycles appear in Wang and Young \cite{WY03}, Lin and Young \cite{LY08} and Engel, Lamb and  Rasmussen \cite{ELR}.

\subsection{Structure of the paper}

 The proof follows as closely as possible the proof in \cite{BG02} of the corresponding result with small multiplicative noise $\e \sigma x_t \circ dW_t$ in place of small additive noise $\e \sigma \circ dW_t$.   While the action-angle formulation used in \cite{AIN04} has theoretical elegance, the use of the explicit Hamiltonian frame described in \eqref{change} together with the original $(x,y)$ coordinate system allows for explicit computations near the singular point $(0,0)$ for the Hamiltonian system.   
  Many of the changes from \cite{BG02} are simple computational consequences of using a different noise vector field $V_1$, resulting in slightly different formulas.  
  Other changes are of more theoretical significance, and these issues are discussed in Section \ref{sec sing}. 

The existence and uniqueness of the almost-sure limit in \eqref{lam ab} is given in Section \ref{sec FK}.  The integral formula \eqref{FK} for the Lyapunov exponent is obtained using the standard Cartesian coordinate expression for the derivative process $\{v_t: t \ge 0\}$.  But to see the $\e^{2/3}$ asymptotic it is necessary to consider the components of $\{v_t: t \ge 0\}$ relative to a moving frame determined by the Hamiltonian function $H(x,y)$.  This is carried out in Sections \ref{sec Ham frame} and \ref{sec AIN}, following the method used in \cite{BG02}.  
 Section \ref{sec Ham frame} deals with the general system \eqref{sde} and includes in Section \ref{sec heur} a heuristic argument for the formulas in Theorem \ref{thm main}. Section \ref{sec AIN} contains the calculations specific to the system \eqref{add eps}.    

The use of the Hamiltonian frame introduces a singularity for the derivative process when $(x_t,y_t) = (0,0)$.  The consequences of this singularity are discussed in Section \ref{sec sing}, and a truncation method is introduced.  A truncated version of the Furstenberg-Khas'minskii formula is given as Theorem \ref{thm lamQ}, and a truncated version of the adjoint method is given as Theorem \ref{thm QG}.  The proof of Theorem \ref{thm main} is completed in Section \ref{sec proof}; and finally Appendix \ref{sec appendix} contains results on the parabolic H\"{o}rmander condition and controllability for the derivative process for a general system of the form \eqref{nonlin}.

\section{Furstenberg-Khasminskii formula} \label{sec FK}

Write $v_t = \|v_t\|\begin{bmatrix} \cos \theta_t \\ \sin \theta_t \end{bmatrix}$.  For the system \eqref{add eps} we have
  \begin{equation} \label{log v}
   d \log \|v_t\| = (-3x_t^2 \sin \theta_t \cos\theta_t - \e^2 \beta \sin^2 \theta_t)dt
  \end{equation}
and 
 \begin{equation} \label{theta}
   d \theta_t = \widetilde{Q}(x_t,y_t,\theta_t)dt
  \end{equation}  
where
  \begin{equation} \label{Q tilde}
   \widetilde{Q}_\e(x,y,\theta) =  -3x^2 \sin \theta \cos\theta - \e^2 \beta \sin^2 \theta.
   \end{equation}

\begin{proposition} \label{prop FK cart}  For the system \eqref{add eps} and for all $(x_0,y_0,v_0)$ with $v_0 \neq 0$ the limit
      $$
       \lambda(\e^2 \beta,\e \sigma,1,1) = \lim_{t \to \infty} \frac{1}{t} \log \|v_t\|
      $$
exists almost surely and has the value
   \begin{equation} \label{FK}
     \lambda(\e^2 \beta,\e \sigma,1,1) = \int \widetilde{Q}_\e(x,y,\theta)d\nu_\e(x,y,\theta)
     \end{equation}
where $\nu_\e$ is the unique invariant probability measure for the process $\{(x_t,y_t,\theta_t): t \ge 0\}$ on $\R^2 \times \R/(2\pi \Z)$ given by (\ref{add eps} ,\ref{theta}).
\end{proposition}  
 
 \proof  We follow the scheme of \cite[Section 8]{BSri}.  The existence of the invariant probability $\mu$ for $\{(x_t,y_t): t \ge 0\}$ together with the compactness of $\R/(2 \pi \Z)$ implies the existence of at least one invariant probability measure $\nu_\e$ for the process $\{(x_t,y_t,\theta_t): t \ge 0\}$ on $\R^2 \times \R/(2\pi \Z)$. 
    The parabolic H\"{o}rmander condition proved in Lemma \ref{lem Hor} implies $\nu_\e$ has a smooth density (see \cite[Theorem 4]{IK}), and the controllability proved in Lemma \ref{lem cont} together with the support theorem of Stroock and Varadhan \cite{SVsupp} imply $\nu_\e(U) > 0$ for all open $U \subset \R^2 \times \R/(2 \pi \Z)$.  It follows that $\nu_\e$ is unique and $\mbox{supp}(\nu_\e) = \R^2 \times \R/(2\pi \Z)$.  Moreover, the time $t$ transition probability $P_t((x,y,\theta), \cdot)$ is absolutely continuous with respect to $\nu_\e$ for all $t > 0$ and $(x,y,\theta) \in \R^2 \times \R/(2 \pi \Z)$, see \cite[Prop 5.1]{IK}.
 
 Integrating \eqref{log v} with respect to $t$ and using the ergodic theorem we get  
  $$
  \frac{1}{t} \log \|v_t\| = \frac{1}{t} \log\|v_0\|+ \frac{1}{t}\int_0^t \widetilde{Q}_\e(x_s,y_s,\theta_s)ds \to \int \widetilde{Q}_\e(x,y,\theta)d\nu_\e(x,y,\theta)
  $$
almost surely as $t \to \infty$ for $\nu_\e$ almost all $(x_0,y_0,\theta_0)$.  Since the time 1 transition probability $P_1((x_0,y_0,\theta_0),\cdot)$ is absolutely continuous with respect to $\nu_\e$ we get  
$$
  \frac{1}{t} \log \|v_t\| = \frac{1}{t} \log\|v_0\|+ \frac{1}{t}\int_0^t \widetilde{Q}_\e(x_s,y_s,\theta_s)ds \to \int \widetilde{Q}_\e(x,y,\theta)d\nu_\e(x,y,\theta)
  $$
almost surely as $t \to \infty$ for all $(x_0,y_0,\theta_0)$, equivalently for all $(x_0,y_0,v_0)$ with $v \neq 0$, giving the Furstenberg-Khasminskii formula \eqref{FK}.
\endproof

 \section{Hamiltonian frame for the derivative process} \label{sec Ham frame}

This section is based very closely on \cite[Section 2]{BG02} and deals with the general system \eqref{sde}.  Recall the orthogonal vector fields 
$U_1$ and $U_2$ on $\R^2 \setminus \{(0,0)\}$ given in \eqref{U}.     
For $(x_t,y_t) \neq (0,0)$ let $w_t = \begin{bmatrix} w_{1,t} \\ w_{2,t} \end{bmatrix}$ denote the coordinates of the derivative process $v_t$ with respect to the moving frame given by
$\Gamma(H(x,y))U_1(x,y)$ and $U_2(x,y)$. Here $\Gamma$ is a smooth
positive function we shall choose later, see \eqref{Gamma}.  (Note that in this paper $\Gamma(h)$ is not the usual Gamma function, with the only exception appearing in the formula for $\gamma_0$ given in Theorem \ref{thm main}.)   More precisely, we use the formula 
 \begin{equation}
   v_t = w_{1,t}\Gamma(H(x_t,y_t))U_1(x_t,y_t)+ w_{2,t}U_2(x_t,y_t).
   \label{change}
 \end{equation}
to convert $v_t$ to the new vector $w_t$.  

For the Hamiltonian flow obtained by taking $\e = 0$ in \eqref{sde} we have 
    \begin{equation} \label{dw0} 
    dw_t  = \begin{bmatrix} 0 & \displaystyle{\frac{J(x_t,y_t)}{\Gamma(H(x_t,y_t))}} \\[1ex]
   0 & 0 \end{bmatrix}w_t\,dt
   \end{equation}
where (using subscripts $x$ and $y$ to denote partial derivatives) 
  \begin{equation}  \nonumber  
    J(x,y) = \frac{[H_y(x,y)^2 -
H_x(x,y)^2][H_{yy}(x,y) - H_{xx}(x,y)] +
4H_x(x,y)H_y(x,y)H_{xy}(x,y)}{[H_x(x,y)^2+H_y(x,y)^2]^2}.
\end{equation}
See \cite[Lemma 1]{BG02}.  
This nilpotent structure is key to everything which follows.  In \cite{AIN04} the nilpotent structure is seen by using action-angle coordinates.  In this paper, as in \cite{BG02}, we use the original $(x,y)$ coordinate system together with the explicit formulas for the moving frame given by
$\Gamma(H(x,y))U_1(x,y)$ and $U_2(x,y)$.

For the stochastically perturbed system \eqref{sde} with $\e >0$ define functions $\alpha_i^j(x,y)$ for $(x,y) \neq (0,0)$ by 
     $$
    V_i(x,y) = \alpha_i^1(x,y) U_1(x,y) + \alpha_i^2(x,y)U_2(x,y)
    $$
for $i =0,1$.  

\begin{lemma} \label{lem dw} \cite[Lemma 3]{BG02}
For $(x_t,y_t) \neq (0,0)$ and $\e > 0$   
 \begin{equation} \label{dw}
 dw_t  = \begin{bmatrix} 0 & \displaystyle{\frac{J(x_t,y_t)}{\Gamma(H(x_t,y_t))}} \\[1ex]
   0 & 0 \end{bmatrix}w_t\,dt + \varepsilon^2 M_0(x_t,y_t)w_t\,dt +
  \varepsilon M_1(x_t,y_t)w_t\circ  dW_t,
 \end{equation}
 where
 \begin{align} \nonumber
  \lefteqn{M_i(x,y)}\\
  & = \begin{bmatrix}
       (U_1.\alpha_i^1)(x,y)-[J(x,y)+(\log \Gamma)'(H(x,y))]\alpha_i^2(x,y)
        &  \displaystyle{\frac{(U_2.\alpha_i^1)(x,y)+
          J(x,y)\alpha_i^1(x,y)}{\Gamma(H(x,y))}}
          \\[3ex]
         \Gamma(H(x,y))(U_1.\alpha_i^2)(x,y) & (U_2.\alpha_i^2)(x,y)
         \end{bmatrix} \label{Mi}
 \end{align}
for $i = 0,1$.  \end{lemma}

\s

Since the sde \eqref{dw} is a small perturbation of a nilpotent linear system we follow the method of Pinsky and Wihstutz \cite{PW88} and define $\widetilde{w}_t = T w_t$ where $T = \begin{bmatrix} \varepsilon^{2/3} & 0 \\ 0 & 1 \end{bmatrix}$.   

\begin{lemma} \label{lem dweps}
For $(x_t,y_t) \neq (0,0)$ and $\e > 0$ 
  \begin{equation} \label{dweps}
 d\widetilde{w}_t  = \begin{bmatrix} 0 & \displaystyle{\frac{\e^{2/3}J(x_t,y_t)}{\Gamma(H(x_t,y_t))}} \\[1ex]
   0 & 0 \end{bmatrix}\widetilde{w}_t\,dt + \varepsilon^2 M_0^\e(x_t,y_t)\widetilde{w}_t\,dt +
  \varepsilon M_1^\e(x_t,y_t)\widetilde{w}_t\circ  dW_t,
 \end{equation}
 where 
\begin{align}
\lefteqn{M_i^\varepsilon(x,y) = TM_i(x,y)T^{-1}}\nonumber \\
   & = \begin{bmatrix}
       (U_1.\alpha_i^1)(x,y)-[J(x,y)+(\log \Gamma)'(H(x,y))]\alpha_i^2(x,y)
        &  \displaystyle{\frac{\e^{2/3}[(U_2.\alpha_i^1)(x,y)+
          J(x,y)\alpha_i^1(x,y)]}{\Gamma(H(x,y))}}
          \\[3ex]
         \e^{-2/3}\Gamma(H(x,y))(U_1.\alpha_i^2)(x,y) & (U_2.\alpha_i^2)(x,y)
         \end{bmatrix}.  \label{Mi eps}
\end{align}
for $i = 0, 1$.
\end{lemma}

For the proof of Lemma \ref{lem dw} see \cite[Lemma 3]{BG02}, and then Lemma \ref{lem dweps} follows immediately.

 \begin{remark} \label{rem frame} The frame is not defined at $(x,y) = (0,0)$, so that we have introduced a singularity into the system.   Since $v_t$ is well-defined for all $t\ge 0$, then $w_t $ and $\widetilde{w}_t$ are well-defined whenever $(x_t,y_t) \neq (0,0)$.  The truncation method introduced in Section \ref{sec trunc} ensures that it is enough to consider processes of the form $\{F(x_t,y_t,\widetilde{w}_t): t \ge 0\}$ where $F(x,y,\widetilde{w}) = 0$ whenever $x^2+y^2 < \delta$ for some $\delta > 0$.  For these purposes It\^{o}'s formula can be applied to prove Lemmas \ref{lem dw} and \ref{lem dweps} giving an expression for $d\widetilde{w}_t$ valid whenever $(x_t,y_t) \neq (0,0)$.  
  And similarly for the calculations in Lemma \ref{lem phi logw} and the expressions \eqref{dLam} and \eqref{LFeps} and \eqref{dFeps}.

\end{remark}

\subsection{Heuristic argument for Theorem \ref{thm main}} \label{sec heur}

At this point we give a heuristic argument to show the significance of the nilpotent structure in \eqref{dw0}.  We can rewrite \eqref{dweps} as 
     \begin{align}
 d\widetilde{w}_t &  = \e^{2/3}\begin{bmatrix} 0 & \displaystyle{\frac{J(x_t,y_t)}{\Gamma(H(x_t,y_t))}} \\[1ex]
   0 & 0 \end{bmatrix}\widetilde{w}_t\,dt +  \e^{1/3} \begin{bmatrix} 0 & 0 \\[1ex]
   \Gamma(H(x_t,y_t))(U_1.\alpha_1^2)(x_t,y_t) & 0 \end{bmatrix} \widetilde{w}_t \circ dW_t  \nonumber \\
   & \quad + (\mbox{terms of order }\e^{4/3} \mbox{ and higher})\widetilde{w}_tdt \nonumber \\
   & \quad +(\mbox{terms of order }\e \mbox{ and higher})\widetilde{w}_t \circ dW_t. \label{dweps hot}
   \end{align}
Recall \eqref{sde}
  $$
    d(x_t,y_t) = \overline{\nabla} H(x_t,y_t)dt + \e^2 V_0(x_t,y_t)dt+ \e V_1(x_t,y_t) \circ dW_t.
   $$
As  $\e \to 0$ there is a separation of time scales in the equations \eqref{sde} and \eqref{dweps hot}. The $(x_t,y_t)$ process moves around the orbits of the
Hamiltonian system at rate 1; the vector $\widetilde{w}_t$ moves at the slower rate $\e^{2/3}$; and the
process $h_t = H(x_t,y_t)$ moves at the even slower rate $\e^2$.  The existence of two different time scales
for motion around orbits and motion across orbits is the basis of many results on
stochastic averaging for small perturbations of Hamiltonian systems, see Khas’minskii \cite{Khas64} and Freidlin and Wentzell \cite{FW94}.  
Here we have also the process $\widetilde{w}_t$
running at a third intermediate rate.
We will average around each of the orbits of the unperturbed
Hamiltonian system using the measure $m_h$, then average over $\widetilde{w}_t$, and finally average over $h_t$.  

For this heuristic argument we ignore the higher order terms in \eqref{dweps hot}.  Averaging around the Hamiltonian orbits involves taking the $m_h$ integral of the coefficients in the $\e^{2/3} \cdots  dt $ term and the root mean square $m_h$ integral of the coefficients in the $\e^{1/3} \cdots dW_t$ term.  The leading terms in \eqref{dweps hot} become
        \begin{equation}
 d\widetilde{w}_t   = \e^{2/3}\begin{bmatrix} 0 & \displaystyle{\frac{G_1(h_t)}{\Gamma(h_t)}} \\[1ex]
   0 & 0 \end{bmatrix}\widetilde{w}_t\,dt +  \e^{1/3} \begin{bmatrix} 0 & 0 \\[1ex]
   \Gamma(h_t)\sqrt{G_2(h_t)} & 0 \end{bmatrix} \widetilde{w}_t \circ dW_t  \label{dweps hot2}
   \end{equation}
where the functions $G_1(h)$ and $G_2(h)$ are given in (\ref{G1},\ref{G2}).  Here we use 
  \begin{equation} \label{int J}
  \int_{H^{-1}(h)} J(x,y)dm_h(x,y) = -\frac{T'(h)}{T(h)} = G_1(h),
  \end{equation} see \cite[Lemma 5]{BG02}, and $V_1.H(x,y) = \alpha_1^2(x,y)$.  
We now choose 
   \begin{equation} \label{Gamma} 
   \Gamma(h) = \left(\frac{G_1(h)}{G_2(h)}\right)^{1/3}
   \end{equation}
so that    
   \begin{equation} \label{G bis} 
    \frac{G_1(h)}{\Gamma(h)} = \Bigl( \Gamma(h) \sqrt{G_2(h)} \Bigr)^2 =  \big|G_1(h)\big|^{2/3} \bigl(G_2(h)\bigr)^{1/3} \equiv G(h).
   \end{equation}
Equation \eqref{dweps hot2} can now be written 
 \begin{equation} \label{dweps hot3}
 d\widetilde{w}_t   = \e^{2/3}G(h_t)\begin{bmatrix} 0 & 1 \\[1ex]
   0 & 0 \end{bmatrix}\widetilde{w}_t\,dt +  \e^{1/3} \sqrt{G(h_t)} \begin{bmatrix} 0 & 0 \\[1ex]
   1 & 0 \end{bmatrix} \widetilde{w}_t \circ dW_t 
   \end{equation}
and this is the time change by a factor $\e^{2/3} G(h_t)$ of the constant coefficient stochastic differential equation 
   \begin{equation} \label{X}
   dX_t   = \begin{bmatrix} 0 & 1 \\[1ex]
   0 & 0 \end{bmatrix}X_t\,dt +  \begin{bmatrix} 0 & 0 \\[1ex]
   1 & 0 \end{bmatrix} X_t  \circ dW_t
   \end{equation}          
The system \eqref{X} has (top) Lyapunov exponent $\gamma_0 \approx 0.29$.  The exact value of $\gamma_0$ given in Theorem \ref{thm main} is due to  Ariaratnam and Xie \cite{AX90}. 
Temporarily freezing the slowest variable $h_t = h$ gives a Lyapunov exponent $\e^{2/3} \gamma_0 G(h)$ for the equation \eqref{dweps hot3}.  Finally, averaging over the slowest variable $h_t = H(x_t,y_t)$ gives 
   $$
   \lambda(\e) \sim \e^{2/3} \gamma_0 \int_{\R^2} G(H(x,y))d\mu(x,y).
   $$

This is not a proof.  A proof based on these ideas would involve two steps.  The first would be to obtain a version of the Furstenberg-Khas'minskii formula \eqref{FK} using $\log \|\widetilde{w}_t\|$ in place of $ \log \|v_t\|$, and the second would be to obtain an adjoint expansion for the Furstenberg-Khas'minskii integrand.  Both steps are non-trivial because of the singularity of the Hamiltonian frame at $(0,0)$. 

We continue with calculations for the process away from $(0,0)$ in Section \ref{sec FK Ham} for the general system \eqref{sde} and Section \ref{sec AIN} for the particular system \eqref{add eps}.  In Section \ref{sec sing} we return to the issue of the singularity at $(0,0)$.   We give a truncated version of the Furstenberg-Khas'minskii formula in Theorem \ref{thm lamQ} and a truncated version of the adjoint expansion in Theorem \ref{thm QG}, before completing the proof in Section \ref{sec proof}.

\subsection{Polar coordinates in the Hamiltonian frame} \label{sec FK Ham} 

In Section \ref{sec FK} the vector equation for $v_t$ in the original Cartesian coordinate system gave rise to the equations (\ref{log v},\ref{theta}) for $\log\|v_t\|$ and $\theta_t$. 
Here we carry out a similar calculation for the derived process $\widetilde{w}_t$ given in \eqref{dweps}.   For $\widetilde{w}_t \neq 0$ write $\widetilde{w}_t = \|\widetilde{w}_t\|\left[\begin{array}{c}
           \cos \phi_t \\ \sin \phi_t
        \end{array}\right]$.
\begin{lemma} \label{lem phi logw} {\rm \cite[Lemma 4]{BG02}} For $(x_t,y_t) \neq (0,0)$ and $\e > 0$,
\begin{align}
  d(\log \|\widetilde{w}_t\|) & = 
    \left[\varepsilon^{2/3}Q_0(x_t,y_t,\phi_t)
         + \varepsilon^{4/3}Q_1(x_t,y_t,\phi_t)
         +\varepsilon^{2}Q_2(x_t,y_t,\phi_t)\right.  \nonumber \\[1ex]
    &  \hspace{10ex}   +\left. \varepsilon^{8/3}Q_3(x_t,y_t,\phi_t)
         +\varepsilon^{10/3}Q_4(x_t,y_t,\phi_t) \right] dt \nonumber \\
    &  \quad + \left[\varepsilon^{1/3}  Q_5(x_t,y_t,\phi_t)
            + \varepsilon Q_6(x_t,y_t,\phi_t)
             + \varepsilon^{5/3}  Q_7(x_t,y_t,\phi_t)\right] dW_t \label{log w}
 \end{align}
 and
\begin{align}
  d \phi_t & =     \left[ \varepsilon^{2/3}t_0(x_t,y_t,\phi_t)
         + \varepsilon^{4/3}t_1(x_t,y_t,\phi_t)
         +\varepsilon^{2}t_2(x_t,y_t,\phi_t) \right. \nonumber \\[1ex]
    &  \hspace{10ex} + \left. \varepsilon^{8/3}t_3(x_t,y_t,\phi_t)
         +\varepsilon^{10/3}t_4(x_t,y_t,\phi_t) \right] dt   \nonumber \\
    &  +  \left[\varepsilon^{1/3}  t_5(x_t,y_t,\phi_t)
            + \varepsilon t_6(x_t,y_t,\phi_t)
             + \varepsilon^{5/3}  t_7(x_t,y_t,\phi_t) \right]dW_t
            \label{dphi} 
 \end{align} 
where 

(i) each of the functions $\e^{2/3}Q_0$, $\e^{4/3}Q_1$, $\e^2Q_2$, $\e^{8/3}Q_3$, $\e^{10/3}Q_4$, $\varepsilon^{2/3}t_0$, $\varepsilon^{4/3}t_1$, $\varepsilon^{2}t_2$, $\varepsilon^{8/3}t_3$ and $\varepsilon^{10/3}t_4$ is a sum of terms of the form
$\varepsilon^\kappa f(\phi)g(x,y)$ where each $f(\phi)$ is a trigonometric polynomial and
each $\varepsilon^\kappa g(x,y)$ is either an entry in $ \begin{bmatrix} 0 & \displaystyle{\frac{\e^{2/3}J(x,y)}{\Gamma(H(x,y))}} \\[1ex]
   0 & 0 \end{bmatrix}+\e^2M_0^\e(x,y)$ or an entry in
$\e^2 V_1.M_1^\e(x,y)$ or a product of two entries
in $\e M_1^\e(x,y)$, and

(ii) each of the functions $ \e^{1/3}Q_5$,
$\e Q_6$, $\e^{5/3}Q_7$, $\varepsilon^{1/3}t_5$, $\varepsilon t_6$
and $\varepsilon^{5/3}  t_7$ is a sum of terms of the form
$\varepsilon^\kappa f(\phi)g(x,y)$ where each $f(\phi)$ is a trigonometric polynomial and
each $\varepsilon^\kappa g(x,y)$ is an entry in $\e M_1^\e(x,y)$.

\n In particular 
\begin{equation} \label{Q}
 Q_0(x,y,\phi) = \frac{J(x,y)}{\Gamma(H(x,y))}\cos \phi \sin \phi
       +\Bigl( \Gamma(H(x,y)) (U_1.\alpha_1^2)(x,y)\Bigr)^2
              \left(\frac{1}{2}\cos^2 \phi - \sin^2 \phi \cos^2 \phi \right)
\end{equation}
and
 \begin{equation} \label{t_0}
 t_0(x,y,\phi) =  -\frac{J(x,y)}{\Gamma(H(x,y))}\sin^2 \phi
       - \Bigl(\Gamma(H(x,y)) (U_1.\alpha_1^2(x,y)\Bigr)^2
              \cos^3 \phi  \sin \phi
  \end{equation}
and
  \begin{equation} \label{t_5}
   t_5(x,y,\phi) =  \Gamma(H(x,y)) (U_1.\alpha_1^2)(x,y)
              \cos^2 \phi.
  \end{equation}
\end{lemma}

\proof  This is a more detailed statement than \cite[Lemma 4]{BG02}, although the proof is exactly the same.  It will be useful later because it will enable matching of powers of $\e$ to rates of blow-up of functions $Q_i$ and $t_i$ as $(x,y) \to (0,0)$.
 \endproof

\begin{remark}  To carry out the transformation from $v_t$ to $\widetilde{w}_t$ it is convenient to use Stratonovich stochastic integrals, but in order to obtain a version of the Furstenberg-Khasminskii formula we need to consider It\^{o} stochastic integrals.  The transformation from Stratonovich to It\^{o} was carried out to obtain \eqref{log w} and \eqref{dphi} and it affects the functions $Q_i$ and $t_i$ for $0 \le i \le 4$.  In particular it is the reason why entries in $\e^2 V_1.M_1^\e(x,y)$ and products of two entries in $\e M_1^\e(x,y)$ appear in the formulas for the $Q_i$ and $t_i$ for $0 \le i \le 4$.   
\end{remark}

\section{Estimates for the system \eqref{add eps}}
 \label{sec AIN}

Here we specialize to the system \eqref{add eps} with the Hamiltonian $H(x,y) = x^2/2 + x^4/4 + y^2/2$ and the vector fields
  $$
  V_0(x,y) = \begin{bmatrix} 0 \\ -\beta y \end{bmatrix}, \quad \quad
    V_1(x,y) = \begin{bmatrix} 0 \\ \sigma \end{bmatrix},
$$
with $\beta > 0$ and $\sigma^2 >0$.  We have
 $$
 U_1(x,y) = \overline{\nabla}H(x,y) = \begin{bmatrix} y \\ -x-x^3 \end{bmatrix}, \quad \quad
 U_2(x,y) = \frac{\nabla H(x,y)}{\|\nabla H(x,y)\|^2} =  \frac{1}{(x+x^3)^2+y^2} \begin{bmatrix} x+x^3 \\ y \end{bmatrix}
          $$
and so
 \begin{align*}
  \alpha_0^1(x,y)&  =  \frac{\beta y(x+x^3)}{(x+x^3)^2 + y^2}, \\
   \alpha_0^2(x,y) & =  -\beta y^2,\\
  \alpha_1^1(x,y) &  =  -\frac{\sigma (x+x^3)}{(x+x^3)^2 + y^2},\\
   \alpha_1^2(x,y) & =  \sigma y.
    \end{align*}

 \subsection{Averaging around orbits; the functions $\Gamma(h)$ and $G(h)$} \label{sec orbitave}

The following calculations are small adaptations of those in \cite[Section 3.2]{BG02}.  Recall the functions $T(h)$, $G_1(h)$ and $G_2(h)$ defined in Section \ref{sec main}. 

\begin{lemma} \label{lem elliptic} Suppose $h > 0$.  We have
  \begin{align*}
  T(h) & =  4(1-2m^2)^{1/2} K(m), \\
  G_1(h) & =  -\frac{T'(h)}{T(h)} =  \left( \frac{1}{m(1-2m^2)}-\frac{E(m)}{m(1-m^2)K(m)} 
     \right) \frac{dm}{dh}> 0,  \\[1ex]
  G_2(h)  & =  \sigma^2\left(\frac{8}{15}h + \left(\frac{7}{15}+ \frac{12}{5}h\right)\left((1+4h)^{1/2}\left(\frac{2E(m)}{K(m)} -1\right) - 1\right) \right)
     \end{align*}
where
 $$m^2 = \frac{1}{2}\left( 1 - \frac{1}{\sqrt{4h+1}}\right)
 $$
and $K(m)$ and $E(m)$ denote the complete elliptic integrals of the
first and second kind (see \cite[Sect.\ 22.7]{WW62}).
\end{lemma}

\proof  The formulas for $T(h)$ and $G_1(h)$ depend only on the function $H$ and are given in \cite[Section 3.2]{BG02}.  The fact that $E(m) < K(m)$ for $m \neq 0$ implies $G_1(h) > 0$.  The formula for $G_2(h)$ uses the same method as in the proof of \cite[Corollary 3]{BG02} together with \cite[Eqn (42)]{BG02} but now starting with $V_1.H(x,y) = \sigma y$ so that 
   $$
  G_2(h) = \int_{H^{-1}(h)} \bigl((U_1.(V_1.H))(x,y)\bigr)^2 dm_h(x,y)   =  \sigma^2 \int_{H^{-1}(h)} \left(x+x^3\right)^2\,dm_h(x,y).
   $$
   \endproof

Information about the behavior of the complete elliptic
integrals $K(m)$ and $E(m)$ for $m$ near 0 and $1/\sqrt{2}$ now gives the following result.

\begin{lemma} \label{lem Gi}
(i)  $G_1(h)$ and $G_2(h)$ are analytic functions of $h$ in a
neighborhood of 0.  As $h \to 0$ we have
 \begin{align*}
 G_1(h) & =  \frac{3}{4} - \frac{87h}{32} +
 \frac{657h^2}{64} + O(h^3), \\
   G_2(h) & = h+\frac{15 h^2}{8}-\frac{25 h^3}{16}+ O(h^4).
 \end{align*}

(ii) As $h \to \infty$ we have
 $$
 G_1(h) \sim \frac{1}{4h}, \hspace{3em} G_1'(h) \sim
 -\frac{1}{4h^2},
   \hspace{3em} G_1''(h) \sim \frac{1}{2h^3},
 $$
 and
 $$
 G_2(h) \sim \frac{24}{5}\sigma^2 \gamma_1 h^{3/2}, \hspace{2em} G_2'(h) \sim
\frac{36}{5}\sigma^2 \gamma_1 h^{1/2},
   \hspace{2em} G_2''(h) \sim  \frac{18}{5}\sigma^2 \gamma_1 h^{-1/2},
 $$
where $\gamma_1 = 2E(2^{-1/2})/K(2^{-1/2}) - 1 \approx 0.457$.
\end{lemma}

The formulas $\Gamma(h) = \bigl(G_1(h)/G_2(h)\bigr)^{1/3}$
and $G(h) = \bigl(G_1(h)\bigr)^{2/3}\bigl(G_2(h)\bigr)^{1/3}$, see (\ref{Gamma},\ref{G}), now give the following result.

\begin{lemma} \label{lem GG}
(i) As $h \to 0$ we have 
 $$
  \Gamma(h)  \sim   (3/4 \sigma^2)^{1/3}h^{-1/3}, \hspace{3em} \frac{\Gamma'(h)}{\Gamma(h)} \sim  -\frac{1}{3h}, \hspace{3em}  \frac{\Gamma''(h)}{\Gamma(h)} \sim  \frac{4}{9h^2},
 $$
and as $h \to \infty$
  $$
   \Gamma(h)  \sim  (5/96\gamma_1 \sigma^2 )^{1/3}h^{-5/6}, \hspace{3em} \frac{\Gamma'(h)}{\Gamma(h)} \sim   -\frac{5}{6h}, \hspace{3em} \frac{\Gamma''(h)}{\Gamma(h)} \sim \frac{55}{36h^2}.
   $$

(ii) As $h \to 0$ we have 
 $$
  G(h)  \sim   (3\sigma/4)^{2/3}h^{1/3}, \hspace{3em} \frac{G'(h)}{G(h)} \sim  \frac{1}{3h}, \hspace{3em}  \frac{G''(h)}{G(h)} \sim  -\frac{2}{9h^2},
 $$
and as $h \to \infty$
  $$
   G(h)  \sim  (3 \gamma_1 \sigma^2/10)^{1/3}h^{-1/6}, \hspace{3em} \frac{G'(h)}{G(h)} \sim   -\frac{1}{6h}, \hspace{3em} \frac{G''(h)}{G(h)} \sim \frac{7}{36h^2}.
   $$
   \end{lemma}

\subsection{Asymptotics for the $Q_i$ and $t_i$}

 Recall the matrices $M_i(x,y)$ in \eqref{Mi}.

\begin{lemma}  \label{lem Mest} (i) As $r \to 0$ we have
         $$
  M_0(x,y)  =  \begin{bmatrix}
        {\cal O}(1)
        &   {\cal O}(r^{2/3})
          \\[1ex]
          {\cal O}(r^{4/3}) &  {\cal O}(1)
         \end{bmatrix} 
         $$
and  
 $$
  M_1(x,y)  =  \begin{bmatrix}
        {\cal O}(r^{-1})
        &   {\cal O}(r^{-7/3})
          \\[1ex]
          {\cal O}(r^{1/3}) &  {\cal O}(r^{-1})
         \end{bmatrix} 
         $$
 and 
 $$
   V_1.M_1(x,y) = \begin{bmatrix}
        {\cal O}(r^{-2})
        &   {\cal O}(r^{-10/3})
          \\[1ex]
          {\cal O}(r^{-2/3}) &  {\cal O}(r^{-2})
         \end{bmatrix}. 
         $$       

(ii)  As $r \to \infty$ all the entries of $M_0(x,y)$ and $M_1(x,y)$ and $V_1. M_1(x,y)$ have at most polynomial growth. 
  \end{lemma}   

\proof Using the estimates on $\Gamma(h)$ in Lemma \ref{lem GG} together with explicit formulas
for the $\alpha_i^j(x,y)$ and $J(x,y)$ it is a matter of direct calculation to estimate
each of the entries in $M_0(x,y)$ and $M_1(x,y)$ and $V_1. M_1(x,y)$ as $r \to 0$ and $r \to \infty$.  We omit the lengthy details.
\endproof

We can now obtain estimates on the functions $Q_i(x,y,\phi)$ and $t_i(x,y,\phi)$ in Lemma \ref{lem phi logw}.

\begin{lemma} \label{lem Qt} As $r = \sqrt{x^2+y^2} \to 0$
  \begin{equation} \label{Qt growth}
    \begin{split} Q_0(x,y,\phi) \mbox{ and }t_0(x,y,\phi) & \mbox{ are } {\cal O}(r^{2/3}),\\ 
     Q_1(x,y,\phi) \mbox{ and }t_1(x,y,\phi) & \mbox{ are } {\cal O}(r^{-2/3}),\\
     Q_2(x,y,\phi) \mbox{ and }t_2(x,y,\phi) & \mbox{ are } {\cal O}(r^{-2}),\\
     Q_3(x,y,\phi) \mbox{ and }t_3(x,y,\phi) & \mbox{ are } {\cal O}(r^{-10/3}),\\
     Q_4(x,y,\phi) \mbox{ and }t_4(x,y,\phi) & \mbox{ are } {\cal O}(r^{-14/3}),\\
     Q_5(x,y,\phi) \mbox{ and }t_5(x,y,\phi) & \mbox{ are } {\cal O}(r^{1/3}),\\
     Q_6(x,y,\phi) \mbox{ and }t_6(x,y,\phi) & \mbox{ are } {\cal O}(r^{-1}),\\
     Q_7(x,y,\phi) \mbox{ and }t_7(x,y,\phi) & \mbox{ are } {\cal O}(r^{-7/3}).
   \end{split}
   \end{equation}  All functions have at most polynomial growth as $(x,y) \to \infty$.
 \end{lemma}
 
\proof 
Lemma \ref{lem Mest} implies 
            $$
  \e^2M_0^\e(x,y)  =  \begin{bmatrix}
        \e^2{\cal O}(1)
        &   \e^{8/3}{\cal O}(r^{2/3})
          \\[1ex]
          \e^{4/3}{\cal O}(r^{4/3}) &  \e^2{\cal O}(1)
         \end{bmatrix} 
         $$
and  
   $$
  \e M_1^\e(x,y)  =  \begin{bmatrix}
        \e{\cal O}(r^{-1})
        &   \e^{5/3}{\cal O}(r^{-7/3})
          \\[1ex]
          \e^{1/3}{\cal O}(r^{1/3}) &  \e{\cal O}(r^{-1})
         \end{bmatrix} 
         $$
           and
  $$
   \e^2V_1.M_1^\e(x,y) = \begin{bmatrix}
        \e^2{\cal O}(r^{-2})
        &   \e^{8/3}{\cal O}(r^{-10/3})
          \\[1ex]
          \e^{4/3}{\cal O}(r^{-2/3}) &  \e^2{\cal O}(r^{-2})
         \end{bmatrix} 
         $$
  as $r\to 0$, where all the ${\cal O}(r^k)$ terms depend only on $(x,y)$ and not on $\e$.
Also $J(x,y) = {\cal O}(1)$ as $r \to 0$, so that by Lemma \ref{lem GG}(i) we have $J(x,y)/\Gamma(H(x,y)) = {\cal O}(r^{2/3})$.  The result now follows from the characterization of the $Q_i$ and $t_i$ in Lemma \ref{lem phi logw}(i,ii).  \endproof

\begin{remark} \label{rem r e}  In steps (ii) and (v) of the proof of Theorem \ref{thm lamQ} it is important that higher rates of blow-up for the functions $Q_1, \ldots, Q_7$ are associated with higher powers of $\e$ in Lemma \ref{lem phi logw}.  Similarly in steps (i), (ii) and (iv) of the proof of Theorem \ref{thm QG} it is important that higher rates of blow-up for the functions $t_1, \ldots, t_7$ and the associated operators ${\cal M}_1, \ldots, {\cal M}_4$ and ${\cal N}_5, \ldots,{\cal N}_7$ are associated with higher powers of $\e$. 
\end{remark} 
 
 \begin{remark}  \label{rem trunc} The system with multiplicative noise treated in \cite{BG02} had a different $V_1(x,y)$ and hence a different $\Gamma(h)  \sim (3/8)^{1/3} h^{-2/3}$ as $h \to 0$.  Then $M_0(x,y)$ and $M_1(x,y)$ and $V_1.M_1(x,y)$ were all $\begin{bmatrix}
        {\cal O}(1)
        &   {\cal O}(r^{4/3})
          \\[1ex]
          {\cal O}(r^{2/3}) &  {\cal O}(1)
         \end{bmatrix}  $ as $r\to 0$.  So the truncation introduced in Section \ref{sec trunc} was not needed in \cite{BG02}.            
\end{remark}

\section{Singularity at $(0,0)$} \label{sec sing}

\subsection{The process $\{(x_t,y_t,\phi_t): t \ge 0\}$} \label{sec process} Recall that the Hamiltonian frame given by \eqref{change} is undefined at $(x,y) = (0,0)$, and so the process $\{(x_t,y_t,\phi_t): t \ge 0\}$ is well-defined only at times $t$ when $(x_t,y_t) \neq (0,0)$.  
  In the setting of \cite{BG02} with multiplicative noise the point $(0,0)$ was fixed and the set $\R^2 \setminus\{(0,0)\}$ was invariant for the process $\{(x_t,y_t): t \ge 0\}$.  Therefore $\{(x_t,y_t,\phi_t): t\ge 0\}$ was a well-defined diffusion process on $(\R^2 \setminus\{(0,0)\}) \times \R/(2 \pi \Z)$ and the tools of ergodic theory were available.  But with additive noise the point $(0,0)$ is no longer fixed.  We will not investigate the issue of the attainability of $(0,0)$ for the process $\{(x_t,y_t): t \ge 0\}$.  We will not attempt to apply ergodic theory to $\{(x_t,y_t,\phi_t): t\ge 0\}$.  Instead we will use ergodic properties of the underlying process $\{(x_t,y_t): t \ge 0\}$ on $\R^2$ in Lemma \ref{lem ergod} below  to obtain bounds on limiting time averages of Lebesgue and It\^{o} integrals involving integrands of the form $F(x_t,y_t,\phi_t)$ with the property that $F(x,y,\phi) = 0$ whenever $(x,y) \neq (0,0)$.    Recall $\mu$ is the invariant probability measure for the system \eqref{add eps} in $\R^2$, and it does not depend on $\e > 0$.

\begin{lemma} \label{lem ergod}  (i)  Suppose $g: \R^2 \to \R$ is in $L^1(\mu)$.  Then for all $\e > 0$ and all $(x_0,y_0) \in \R^2$
   $$
  \lim_{n \to \infty} \frac{1}{n} g(x_n,y_n) \to 0 \quad \mbox{almost surely}.
   $$

(ii) Suppose the process $\{X_t: t \ge 0\}$ satisfies $|X_t| \le g(x_t,y_t)$ where $g: \R^2 \to [0,\infty)$ is in $L^1(\mu)$.  Then for all $\e > 0$ and all $(x_0,y_0,v_0)$ with $v_0 \neq 0$
 $$
   \limsup_{t \to \infty} \frac{1}{t}\left|  \int_0^t X_s ds \right| \le  \limsup_{t \to \infty} \frac{1}{t} \int_0^t g(x_s,y_s) ds  =  \int_{\R^2} g(x,y)d\mu(x,y) <\infty
   $$    
almost surely.

(iii)  Suppose the process $\{X_t: t \ge 0\}$ is non-anticipating and satisfies $|X_t| \le g(x_t,y_t)$ where $g: \R^2 \to [0,\infty)$ is in $L^2(\mu)$.  Then for all $\e > 0$ and all $(x_0,y_0,v_0)$ with $v_0 \neq 0$
   $$
      \lim_{t \to \infty} \frac{1}{t} \int_0^t X_s dW_s = 0
      $$
 almost surely. 
 \end{lemma}
 
\proof (i) The Borel-Cantelli lemma implies $(1/n)g(x_n,y_n)\to 0$ almost surely for $\mu$ almost all $(x_0,y_0)$.  Since the time 1 transition probability $P_1((x_0,y_0), \cdot)$ is absolutely continuous with respect to $\mu$, it follows that $(1/n)g(x_n,y_n)\to 0$ almost surely for all $(x_0,y_0) \in\R^2$. 
 
 (ii) The first inequality is deterministic, and the second uses the ergodic theorem together the absolute continuity of the time 1 transition probability. 
 
 (iii)  Applying (ii) to $(X_t)^2$ gives
    $$
      \limsup_{t\to\infty} \frac{1}{t} \int_0^t (X_s)^2 ds \le \int_{\R^2} \bigl(g(x,y)\bigr)^2 d\mu(x,y) <\infty
      $$
almost surely, and hence (see for example  \cite[Lemma 3.14]{CE23}, also \cite[Theorem 3.4.6]{KarShr})
    $$
    \lim_{t\to\infty} \frac{1}{t} \int_0^t X_s dW_s = 0
    $$    
almost surely.   \endproof
 
\begin{remark} \label{rem disc} To obtain the continuous time limit $\lim_{t \to \infty} \frac{1}{t}g(x_t,y_t)$ in Lemma \ref{lem ergod}(i) we would need  a stronger integrability condition, see \cite[Prop 4.1.3]{Arn98}.  
 For the purposes of this paper, it is sufficient to have the discrete time limit here and in Proposition \ref{prop equiv} and Theorem \ref{thm QG}.

\end{remark}  

\subsection{Truncation}  \label{sec trunc}

The standard method to obtain the Furstenberg-Khas'minskii formula for $\lim_{t \to \infty} \frac{1}{t}\log \|\widetilde{w}_t\|$ involves integrating the functions $Q_0, Q_1, \ldots, Q_7$ which appear in \eqref{log w}.  In Lemma \ref{lem Qt} we see that several of these functions blow up as $r \to 0$, and so we introduce a truncation.  In order to obtain the desired asymptotics as $\e \to 0$, the truncation has to be carried out in an $\e$-dependent neighborhood of $(0,0)$.

Choose and fix a $C^2$ function $\chi:\R \to [0,1]$ such that 
  \begin{equation} \label{chi}
  \chi(h) = \begin{cases} 0 & \mbox{ if } h \le 1 \\
                     1 & \mbox{ if } h \ge 2.
                     \end{cases} 
      \end{equation}
         There are constants $c_1$ and $c_2$ such that $|\chi'(h)| \le c_1$ and $|\chi''(h)| \le c_2$ for all $h$.  For $\e > 0$ define $\chi_\e(h) = \chi(h/\e )$, and note that $|\chi_\e'(h)| \le c_1/\e$ and $|\chi_\e''(h)| \le c_2/\e^2$. 

\s 

For the system \eqref{add eps} the distribution of the process $\{(x_t,y_t): t \ge 0\}$ depends on $\e$, but the invariant probability measure $\mu$ does not depend on $\e$.  So if $g(x,y)$ does not depend on $\e$ then the integral $\int_{\R^2}g(x,y)d\mu(x,y)$ does not depend on $\e$.
The following lemma relates the truncation parameter $\e$ with the rate of blow-up of $g(x,y)$ as $r  = \sqrt{x^2+y^2}  \to 0$.  It is designed to take advantage of the fact, seen in the estimates \eqref{Qt growth} for the $Q_i$ and $t_i$, that functions associated with higher powers of $\e$ typically have worse blow-up as $r \to 0$.   Recall $h = H(x,y) \sim r^2/2$ as $r \to 0$. 

\begin{lemma} \label{lem gint} Suppose $g: \R^2 \setminus \{(0,0)\} \to [0,\infty)$ satisfies  $\int_{r > \delta} g(x,y)d\mu(x,y) < \infty$ for each fixed $\delta >0$ and $g(x,y) = {\cal O}(r^k)$ as $r \to 0$.  Then   
  \begin{equation} \label{gint}
    \int_{\R^2} \chi_\e(h) g(x,y) d\mu(x,y) \begin{cases}
          \to \int_{\R^2}g(x,y) d\mu(x,y)  & \mbox{ if } k > -2, \\
          = {\cal O}(\log 1/\e) & \mbox{ if } k = -2, \\
          = {\cal O}(\e^{(k+2)/2}) & \mbox{ if } k <-2
          \end{cases}
  \end{equation}  and
   \begin{equation} \label{gint2}
    \int_{\e \le h \le 2\e} g(x,y) d\mu(x,y)  = \begin{cases} {\cal O}(\e^{(k+2)/2}) & \mbox{ if } k \neq -2, \\
                                                              {\cal O}(\log 1/\e)  & \mbox{ if } k = -2                     
                                        \end{cases}
                                         \end{equation}    
as $\e \to 0$. 
\end{lemma}

\proof  This follows easily from the fact that the density $\rho(x,y)$ of $\mu$ satisfies $\rho(x,y)\to C > 0$ as $(x,y) \to (0,0)$, see \eqref{density}. \endproof

\subsection{Equivalence of norms}
The heuristic argument in Section \ref{sec heur} suggests to study $\|\widetilde{w}_t\|$ in place of $\|v_t\|$.   The following definition, which avoids the singularity in $\|\widetilde{w_t}\|$ when $(x_t,y_t)$ is close to $(0,0)$, was motivated by a similar construction in Chemnitz and Engel \cite{CE23}.  For $\e > 0$ and $t \ge 0$ define
   \begin{equation} \label{Lam}
   \Lambda_\e(t) = \bigl(1-\chi_\e(H(x_t,y_t))\bigr)\log\|v_t\| + \chi_\e(H(x_t,y_t)) \log \|\widetilde{w}_t\|.
   \end{equation}

\begin{proposition} \label{prop equiv} For the system \eqref{add eps}, for all $\e > 0$ and all $(x_0,y_0,v_0)$ with $v_0 \neq 0$ 
    $$
        \limsup_{n \to \infty} \frac{1}{n}\Big|\log \|v_n\| -  \Lambda_\e(n)\Big| =0
    $$
almost surely.
  \end{proposition}  
  
\proof  From \eqref{change} we have 
    \begin{align*}
    v_t  & =  w_{1,t}\Gamma(H(x_t,y_t))U_1(x_t,y_t) + w_{2.t}U_2(x_t,y_t)\\
        & = \widetilde{w}_{1,t}\e^{-2/3}\Gamma(H(x_t,y_t))U_1(x_t,y_t) + \widetilde{w}_{2.t}U_2(x_t,y_t)
    \end{align*}
so that 
  $$
  \|v_t\| \le \max\Bigl(\e^{-2/3}\Gamma(H(x_t,y_t))\|\nabla H(x_t,y_t)\|, \|\nabla H(x_t,y_t)\|^{-1}\Bigr)\|\widetilde{w}_t\|, 
     $$ 
and similarly
  $$
   \|\widetilde{w}_t\| \le \max\Bigl(\e^{2/3}(\Gamma(H(x_t,y_t)))^{-1}\|\nabla H(x_t,y_t)\|^{-1}, \|\nabla H(x_t,y_t)\|\Bigr)\|v_t\|
 $$
It follows that 
   \begin{equation} \label{equiv}
  \Big|\log \|\widetilde{w}_t\| - \log\|v_t\|  \Big| \le \frac{2}{3}|\log \e|  + \Big| \log \Gamma(H(x_t,y_t)) \Big|+ \Big|\log \|\nabla H(x_t,y_t)\|\Big|,
 \end{equation}    
and so
  \begin{align*} \Bigl| \Lambda_\e(t) - \log\|v_t\|\Bigr| 
   & \le \chi_\e(h_t)\Bigl| \log \| \widetilde{w}_t\| - \log \|v_t\| \Bigr|\\
   & \le \chi_\e(h_t)\left(\frac{2}{3}|\log \e| + \Big|\log \|\nabla H(x_t,y_t)\|\Big| + \Big| \log \Gamma(H(x_t,y_t)) \Big|\right).
   \end{align*}  
Using the exact formula for $\|\nabla H(x,y)\|$ and the asymptotics in Lemma \ref{lem GG}(i) for $\Gamma(h)$, we see the right side is integrable with respect to the stationary measure $\mu$, and so
  \begin{align*}
  \lim_{n \to\infty} \frac{1}{n} \Bigl| \Lambda_\e(n) - \log\|v_n\|\Bigr| & \le \lim_{n \to\infty} \frac{1}{n}\chi_\e(h_n)\left(\frac{2}{3}|\log \e| + \Big|\log \|\nabla H(x_n,y_n)\|\Big| + \Big| \log \Gamma(H(x_n,y_n)) \Big|\right)\\
  & = 0
  \end{align*}
almost surely by Lemma \ref{lem ergod}. \endproof

\section{Truncated Furstenberg-Khas'minskii formula} \label{sec FK Ham tr}

For the following result note that $(x_0,y_0,v_0)$ with $v_0 \neq 0$ determines the law of the restriction of $\{(x_t,y_t,\phi_t): t \ge 0\}$ to the set of times when $(x_t,y_t) \neq (0,0)$ and hence determines the laws of $\{\Lambda_\e(t): t \ge 0\}$ and $\{\chi_\e(h_t)Q_0(x_t,y_t,\phi_t): t \ge 0\}$.  Recall the formulas \eqref{Lam} for $\Lambda_\e(t)$ and \eqref{Q} for $Q_0(x,y,\phi)$.

\begin{theorem}\label{thm lamQ}  For the system \eqref{add eps} there exists $K_1$ such that for all $0 < \e \le 1$ and $(x_0,y_0,v_0)$ with $v_0 \neq 0$  
  $$
  \limsup_{t \to \infty} \left| \frac{1}{t} \Lambda_\e(t) - \frac{\e^{2/3}}{t} \int_0^t \chi_\e(h_s)Q_0(x_s,y_s,\phi_s)ds \right|  \le K_1 \e^{4/3}
  $$
almost surely. 
 \end{theorem}

\proof
 Write $H(x_t,y_t) = h_t$.  From \eqref{sde} we have 
  \begin{align*}
  dh_t & =  \e^2 V_0.H(x_t,y_t)dt+ \e (V_1 . H)(x_t,y_t)\circ dW_t \\
    & = \e^2 \overline{\cal L}H(x_t,y_t)dt + \e (V_1 . H)(x_t,y_t) dW_t
  \end{align*}
where 
   \begin{equation} \label{overL}
   \overline{\cal L}f(x,y)= V_0.f(x,y) + \frac{1}{2}V_1^2.f(x.y)
   \end{equation}
and so
   \begin{align} \nonumber 
     d \chi_\e(h_t) & =  \chi_\e'(h_t)dh_t+ \frac{1}{2}\chi_\e''(h_t)(dh_t)^2\\
         & = \e^2\left(\overline{\cal L}H(x_t,y_t)\chi_\e'(h_t) + \frac{1}{2}\bigl((V_1.H)(x_t,y_t)\bigr)^2\chi_\e''(h_t) \right)dt + \e  (V_1.H)(x_t,y_t) \chi'_\e(h_t) dW_t \label{dchi}
         \end{align}             
Applying \eqref{log v}, \eqref{log w} and \eqref{dchi} to the formula \eqref{Lam} for $\Lambda_\e(t)$ we get 
   \begin{align}
     d\Lambda_\e(t)
     & = \Bigl(1-\chi_\e(h_t)\Bigr)d \log\|v_t\| + \chi_\e(h_t)d \log \|\widetilde{w}_t\| \nonumber \\
     & \quad + \Bigl(\log \|\widetilde{w}_t\| - \log \|v_t\|\Bigr)d\chi_\e(h_t) + \Bigl(d \log \|\widetilde{w}_t\|-d \log \|v_t\| \Bigr)d\chi_\e(h_t)  \nonumber  \\
     & = \Bigl(1-\chi_\e(h_t)\Bigr)\widetilde{Q}_\e(x_t,y_t,\theta_t)dt  \nonumber \\
     & \quad + \chi_\e(h_t)\Bigl(\varepsilon^{2/3}Q_0(x_t,y_t,\phi_t)
         + \varepsilon^{4/3}Q_1(x_t,y_t,\phi_t)
         +\varepsilon^{2}Q_2(x_t,y_t,\phi_t)\Bigr.  \nonumber  \\
    &  \hspace{10ex}   +\Bigl. \varepsilon^{8/3}Q_3(x_t,y_t,\phi_t)
         +\varepsilon^{10/3}Q_4(x_t,y_t,\phi_t) \Bigr) dt  \nonumber  \\
    &  \quad + \chi_\e(h_t)\Bigl(\varepsilon^{1/3}  Q_5(x_t,y_t,\phi_t)
            + \varepsilon Q_6(x_t,y_t,\phi_t)
             + \varepsilon^{5/3}  Q_7(x_t,y_t,\phi_t)\Bigr) dW_t  \nonumber \\
     & \quad + \Bigl(\log \|\widetilde{w}_t\| - \log \|v_t\|\Bigr)\e^2\Bigl(\overline{\cal L}H(x_t,y_t)\chi_\e'(h_t) + \frac{1}{2}\bigl((V_1.H)(x_t,y_t)\bigr)^2\chi_\e''(h_t) \Bigr)dt  \nonumber \\
     & \quad  +\Bigl(\log \|\widetilde{w}_t\| - \log \|v_t\|\Bigr)  \e (V_1.H)(x_t,y_t)\chi_\e'(h_t) dW_t  \nonumber \\
     & \quad + \Bigl(\varepsilon^{1/3}  Q_5(x_t,y_t,\phi_t)
            + \varepsilon Q_6(x_t,y_t,\phi_t)
             + \varepsilon^{5/3}  Q_7(x_t,y_t,\phi_t)\Bigr)\e (V_1.H)(x_t,y_t) \chi'_\e(h_t)dt. \label{dLam}
 \end{align}
Integrate from $0$ to $t$ and divide by $t$. On the left we have one term
    $$
    \frac{1}{t}(\Lambda_\e(t) - \Lambda_\e(0)).
    $$
On the right we pick out the term
     $$
     \e^{2/3} \frac{1}{t} \int_0^t \chi_\e(h_s)Q_0(x_s,y_s,\phi_s)ds.
     $$
All the other terms are of the form
    $$
    \frac{1}{t} \int_0^t X_s ds \quad \mbox{ or } \quad \frac{1}{t} \int_0^t X_s dW_s
    $$
and we carry out a term by term analysis of the these remaining terms as $t \to \infty$.  

(i) The term involving $\widetilde{Q}_\e$.  If $1-\chi_\e(h) \neq 0$ then $x^2+y^2 \le 2h \le 4\e $, and so from \eqref{Q tilde} we have $|\widetilde{Q}_\e(x,y,\theta)| \le 3x^2/2+ \e^2 \beta \le 6 \e + \e^2 \beta$.  Then 
   \begin{align*}
  \limsup_{t \to \infty}\frac{1}{t}\left| \int_0^t \bigl(1-\chi_\e(h_s)\bigr)\widetilde{Q}_\e(x_s,y_s,\theta_s)ds \right|
   & \le  \limsup_{t \to \infty} \frac{1}{t}\int_0^t \bigl(1-\chi_\e(h_s)\bigr)(6\e +\e^2 \beta)ds \\
     & = \int_{\R^2} \bigl(1-\chi_\e(h)\bigr)(6 \e + \e^2 \beta) d\mu(x,y)\\
         & \le (6\e +\e^2 \beta) \mu(h \le 2\e)\\
         &  = {\cal O}(\e^2),
     \end{align*}  
see \eqref{density}.

(ii) Terms involving $Q_1, \ldots, Q_4$.  Write $\overline{Q}_i(x,y) = \sup\{|Q_i(x,y,\phi)|: \phi \in \R/2\pi\Z\}$. Using the estimates \eqref{Qt growth} together with Lemmas \ref{lem ergod}(ii) and \ref{lem gint}, we get
\begin{align*}
 \lefteqn{ \limsup_{t \to \infty} \frac{1}{t}\left|\int_0^t \chi_\e(h_s) Q_i(x_s,y_s,\phi_s)ds \right|} \hspace{15ex}\\
 & \le \int_{\R^2} \chi_\e(h)\overline{Q}_i(x,y)\mu(x,y) \,\, \begin{cases}
          \to \int_{\R^2}\overline{Q}_1(x,y)d\mu(x,y) <\infty & \mbox{ if } i=1, \\
          = {\cal O}(\log (1/\e)) & \mbox{ if } i=2, \\
          = {\cal O}(\e^{-2/3}) & \mbox{ if } i=3, \\
           ={\cal O}(\e^{-4/3}) & \mbox{ if } i=4 
           \end{cases}
  \end{align*}
as $\e\to 0$.  Therefore we have
  \begin{align*}
  \limsup_{t \to \infty} \frac{1}{t}\left|\int_0^t \chi_\e(h_s) \e^{4/3}Q_1(x_s,y_s,\theta_s)ds \right| & = {\cal O}(\e^{4/3}),\\
  \limsup_{t \to \infty}\frac{1}{t} \left|\int_0^t \chi_\e(h_s) \e^2 Q_2(x_s,y_s,\theta_s)ds \right| & = {\cal O}(\e^2 \log(1/\e)),\\
   \limsup_{t \to \infty} \frac{1}{t}\left|\int_0^t \chi_\e(h_s) \e^{8/3}Q_3(x_s,y_s,\theta_s)ds  \right|& = {\cal O}(\e^2), \\
  \limsup_{t \to \infty}\frac{1}{t} \left|\int_0^t \chi_\e(h_s) \e^{10/3}Q_4(x_s,y_s,\theta_s)ds  \right|& = {\cal O}(\e^2)
  \end{align*}    
almost surely.

(iii) Stochastic integrals involving $Q_5,\ldots,Q_7$.  Each $Q_i$ is continuous on $(\R^2 \setminus \{(0,0\}) \times \R/(2 \pi \Z)$ and has at most polynomial growth as $r \to \infty$.  So for each $\e > 0$ we have $\chi_\e(h) \overline{Q}_i(x,y) \in L^2(\mu)$ for $i = 5,6,7$ and hence by Lemma \ref{lem ergod}(ii)
    $$
    \lim_{t \to \infty}\frac{1}{t} \int_0^t \chi_\e(h_s) Q_i(x_s,y_s,\phi_s)dW_s = 0
  $$
almost surely for $i = 5,6,7$.

(iv) The two terms involving $\log \|\widetilde{w}_t\| - \log \|v_t\|$.  Recall \eqref{equiv}
    $$
      \Bigl|\log \|\widetilde{w}_t\| - \log \|v_t\|\Bigr| \le \frac{2}{3}|\log \e| + \Big|\log \|\nabla H(x_t,y_t)\|\Big| + \Big| \log \Gamma(H(x_t,y_t)) \Big| .   
      $$
For any $h$ such that $\chi_\e'(h) \neq 0$ or $\chi_\e''(h) \neq 0$ we have $\e \le h \le 2\e$.  Recall $0 \le \e \chi_\e'(h) \le c_1$ and $\e^2|\chi_\e''(h)| \le c_2$.  We have an explicit formula for $\|\nabla H(x,y)\|^2 \sim 2h$ as $h \to 0$ and in Lemma \ref{lem GG} we have the asymptotic $\Gamma(h) \sim (3/4\sigma^2)^{1/3} h^{-1/3}$ as $h \to 0$.  It follows that the exist constants $C_1$ and $C_2$ such that 
     $$
     \sup\Bigl\{\frac{2}{3}|\log \e| + \Big|\log \|\nabla H(x,y)\|\Big| + \Big| \log \Gamma(H(x,y) \Big|: \e  \le H(x,y) \le 2\e \Bigr\} \le C_1 |\log \e| + C_2
     $$
for $0 <\e \le 1/2$.   Using Lemma \ref{lem ergod}(ii), and then the estimates $\overline{\cal L}H(x,y) = -\beta y^2+ \sigma^2/2 = {\cal O}(1)$ and $V_1.H(x,y) = \sigma y = {\cal O}(r)$ as $r \to 0$ in \eqref{gint2} of Lemma \ref{lem gint} we have      
  \begin{align*}
  \lefteqn{\limsup_{t \to \infty}\frac{1}{t}\left|\int_0^t\Bigl(\log \|\widetilde{w}_s\| - \log \|v_s\|\Bigr)\e^2\left(\overline{\cal L}H(x_s,y_s)\chi_\e'(h_s) + \frac{1}{2}\bigl((V_1.H)(x_s,y_s)\bigr)^2\chi_\e''(h_s) \right)ds\right|}\hspace{5ex} \\
  & \le \limsup_{t \to\infty}\frac{1}{t}\int_0^t \bigl(C_1 |\log \e|+C_2\bigr){\bf 1}_{[\e,2\e]}(h_s)\left(\e c_1|\overline{\cal L}H(x_s,y_s)|+ \frac{1}{2}c_2 \bigl((V_1.H)(x_s,y_s)\bigr)^2\right)ds\\
   & = (C_1 |\log \e|+C_2)\int_{\e \le h \le 2\e}\left(\e c_1|\overline{\cal L}H(x,y)|+ \frac{1}{2}c_2\bigl((V_1.H)(x,y)\bigr)^2\right)
   d\mu(x,y)\\
   & = (C_1 |\log \e|+C_2)\Bigl( \e {\cal O}(\e) + {\cal O}(\e^2)\Bigr) = {\cal O}(\e^2 \log 1/\e)
  \end{align*}  
as $\e \to 0$.
  Similarly for the stochastic integral involving $\log \|\widetilde{w}_t\| - \log \|v_t\|$ we have 
   $$
   \left|\Bigl(\log \|\widetilde{w}_t\| - \log \|v_t\|\Bigr)  \e V_1.H(x_t,y_t)\chi_\e'(h_t)\right|
   \le (C_1|\log \e| +C_2) {\bf 1}_{[\e,2\e]}(h_s) c_1 \sigma |y_s|
   $$
and ${\bf 1}_{[\e,2\e]}(h)|y| \in L^2(\mu)$.  It follows by Lemma \ref{lem ergod} that for eac $\e > 0$ 
 $$
 \lim_{t \to \infty} \frac{1}{t} \int_0^t \Bigl(\log \|\widetilde{w}_s\| - \log \|v_s\|\Bigr)  \e (V_1.H)(x_s,y_s)\chi_\e'(h_s) dW_s = 0
 $$
almost surely.

(v) The $dt$ integrals involving $Q_5,\ldots,Q_7$.  Recall $|\e \chi'(h)| \le c_1$ for $\e \le h \le 2\e$ and 0 otherwise.  Also $V_1.H(x,y) = \sigma y = {\cal O}(r)$. From \eqref{Qt growth} we have $\overline{Q}_5(x,y) = {\cal O}(r^{1/3})$ and hence $\overline{Q}_5(x,y) |(V_1.H)(x,y)| = {\cal O}(r^{4/3})$ as $r \to 0$.  Using Lemma \ref{lem ergod}(ii) and then \eqref{gint2} of Lemma \ref{lem gint} we get
    \begin{align*}
\lefteqn{\limsup_{t \to \infty} \frac{1}{t}\left|\int_0^t  \Bigl(\e^{1/3}Q_5(x_s,y_s,\phi_s)\Bigr) \e (V_1.H)(x_s,y_s)\chi'_\e(h_s) ds\right|} \hspace{10ex} \\
  & \le \e^{1/3}\limsup_{t \to \infty} \frac{1}{t}\int_0^t  \Bigl(\overline{Q}_5(x_s,y_s)|(V_1.H)(x_s,y_s)|\Bigr) c_1{\bf 1}_{[\e,2\e]}(h_s)ds \\
  & =\e^{1/3}c_1 \int_{\e \le h \le 2\e} \overline{Q}_5(x,y)\big|(V_1.H)(x,y)\big|\mu(x,y)\\
   & =\e^{1/3} {\cal O}(\e^{5/3}) = {\cal O}(\e^2).
     \end{align*} 
Similarly, using $\overline{Q}_6(x,y) = {\cal O}(r^{-1})$ and $\overline{Q}_7(x,y) = {\cal O}(r^{-7/3})$, we get
   $$
   \limsup_{t \to \infty} \frac{1}{t}\left|\int_0^t  \Bigl(\e Q_6(x_s,y_s,\phi_s)\Bigr) \e (V_1.H)(x_s,y_s)\chi'_\e(h_s) ds\right|  \le  \e {\cal O}(\e) = {\cal O}(\e^2)
   $$
and
 $$
  \limsup_{t \to \infty} \frac{1}{t}\left|\int_0^t  \Bigl(\e^{5/3}Q_7(x_s,y_s,\phi_s)\Bigr) \e (V_1.H)(x_s,y_s)\chi'_\e(h_s) ds\right|  \le \e^{5/3} {\cal O}(\e^{1/3}) = {\cal O}(\e^2).
   $$

Together, all the $dW_t$ integrals have $\lim_{t \to \infty} \frac{1}{t}\int_0^t \cdots dW_s = 0$ almost surely and all the $dt$ integrals except the one with $Q_0$ have bounds on the almost sure value of the $\limsup_{t \to \infty} \frac{1}{t}\int_0^t \cdots ds$ and the bounds are non-random and do not depend on the initial value $(x_0,y_0,v_0)$ with $v_0 \neq 0$.  As $\e \to 0$ the biggest of these bounds is ${\cal O}(\e^{4/3})$ and others are ${\cal O}(\e^2)$ and ${\cal O}(\e^2 \log(1/\e))$.  
This completes the proof of Theorem \ref{thm lamQ}.  \qed

\section{Adjoint method} \label{sec adjoint}

The next step involves the time average of $\chi_\e(h_t)Q_0(x_t,y_t,\phi_t)$.  It is slightly different from the treatment in \cite{BG02} because the invariant probability measure $\mu$ for $\{(x_t,y_t): t \ge 0\}$ given by \eqref{add eps} depends only on the ratio $\e^2\beta/(\e\sigma^2) = \beta/\sigma^2$ and not on $\e$.  The stationary distribution for $\{h_t = H(x_t,y_t): t \ge 0\}$ does not depend on $\e$ and so there is no need to do stochastic averaging for $\{h_t: t \ge 0\}$.  Instead of approximating the adjoint equation for $Q_0(x,y,\phi) - \overline{\lambda}$ (see three lines above \cite[eqn (20)]{BG02}) it suffices to approximate the adjoint equation for $Q_0(x,y,\phi) - \gamma_0 G(h)$, see \eqref{LFeps}.  Therefore the functions $c(x,y)$, $d(x,y)$, $\overline{c}(h)$, $\overline{d}(h)$, the operators ${\cal N}$, $\widetilde{\cal N}$, and the functions $\Psi(h)$, $C(x,y)$ and $D(x,y)$ appearing in \cite[Section 2]{BG02} are not needed here.
      
\subsection{Preliminary calculation} This is based very closely in \cite[Section 2]{BG02}.  Define functions
\begin{equation} \label{ab}
 \begin{split}
 a(x,y) & =  J(x,y)/ \Gamma(H(x,y)),   \\
 b(x,y) & =  \Bigl(\Gamma(H(x,y))(U_1.\alpha_1^2)(x,y)\Bigr)^2. 
 \end{split}
 \end{equation} 
Then
   \begin{equation} \label{Q1}
 Q_0(x,y,\phi) = a(x,y)\cos \phi \sin \phi
       + b(x,y)\left(\frac{1}{2}\cos^2 \phi - \sin^2 \phi \cos^2 \phi \right).
 \end{equation}

  Recall the equations \eqref{sde} for $\{(x_t,y_t): t \ge 0\}$ and \eqref{dphi} for $\{\phi_t: t \ge 0\}$.  With the intention of applying It\^{o}'s formula to the process $\{(x_t,y_t,\phi_t): t \ge 0\}$ define the differential operator ${\cal L}_\varepsilon$ by 
 \begin{align} 
    {\cal L}_\e f & = U_1.f+ \e^2 \overline{\cal L}f +  \Bigl(\varepsilon^{2/3}t_0(z) + \varepsilon^{4/3}t_1(z) +\varepsilon^{2}t_2(z)+ \varepsilon^{8/3}t_3(z) +\varepsilon^{10/3}t_4(z)\Bigr)\frac{\partial f }{\partial \phi} \nonumber  \\
    & \quad 
    + \frac{1}{2}\Bigl(\varepsilon^{1/3}  t_5(z) + \varepsilon t_6(z) + \varepsilon^{5/3}  t_7(z)\Bigr)^2 \frac{\partial^2 f}{\partial \phi^2} \nonumber \\
    & \quad + \Bigl(\e^{4/3}t_5(z)+ \e^2 t_6(z)+\e^{8/3}t_7(z)\Bigr)\frac{\partial }{\partial \phi}(V_1.f) \nonumber \\
   & =: U_1.f+\e^2\overline{\cal L}f + \Bigl( \e^{2/3}{\cal M}_0 + \e^{4/3}{\cal M}_1+\e^2{\cal M}_2+\e^{8/3}{\cal M}_3+ \e^{10/3}{\cal M}_4\Bigr)f \nonumber \\
   & \quad +  \Bigl(\e^{4/3}{\cal N}_5+ \e^2{\cal N}_6+\e^{8/3}{\cal N}_7\Bigr)(V_1.f), \label{Leps}
  \end{align}
say, where $\overline{\cal L}$ is given in \eqref{overL} and each ${\cal M}_i$ is a second order differential operator acting on functions of $\phi$ and each operator ${\cal N}_i$ is a first order differential operator acting on functions of $\phi$.   
More precisely,
   \begin{align*}
  \lefteqn{ \e^{2/3}{\cal M}_0 + \e^{4/3}{\cal M}_1+\e^2{\cal M}_2+\e^{8/3}{\cal M}_3+ \e^{10/3}{\cal M}_4}\hspace{5ex} \\
  & =  \Bigl(\varepsilon^{2/3}t_0(x,y,\phi) + \varepsilon^{4/3}t_1(x,y,\phi) +\varepsilon^{2}t_2(x,y,\phi)+ \varepsilon^{8/3}t_3(x,y,\phi)
         +\varepsilon^{10/3}t_4(x,y,\phi)\Bigr)\frac{\partial }{\partial \phi} \\
    & \quad 
    + \frac{1}{2}\Bigl(\varepsilon^{1/3}  t_5(x,y,\phi) + \varepsilon t_6(x,y,\phi) + \varepsilon^{5/3}  t_7(x,y,\phi)\Bigr)^2 \frac{\partial^2 }{\partial \phi^2}
   \end{align*}
and
    $$
    \e^{4/3}{\cal N}_5+ \e^2{\cal N}_6+\e^{8/3}{\cal N}_7 = \Bigl(\e^{4/3}t_5(x,y,\phi)+ \e^2 t_6(x,y,\phi)+\e^{8/3}t_7(x,y,\phi)\Bigr)\frac{\partial }{\partial \phi}.   
   $$ 
   In particular, using \eqref{t_0} and \eqref{t_5}),
  \begin{equation}\nonumber 
   {\cal M}_0 = -a(x,y) \sin^2 \phi \frac{\partial }{\partial \phi} + b(x,y)\Bigl(-\cos^3 \phi \sin \phi \frac{\partial}{\partial \phi} + \frac{1}{2}\cos^4 \phi \frac{\partial^2}{\partial \phi^2} \Bigr).
   \end{equation}  
It follows from \eqref{Qt growth} that for any given $f \in C^2(\R/2\pi \Z)$ we have
      \begin{equation} \label{MN growth}
    \begin{split}
     {\cal M}_0f(x,y,\phi) &= {\cal O}(r^{2/3})\\ 
     {\cal M}_1f(x,y,\phi) &=  {\cal O}(r^{-2/3})\\
     {\cal M}_2f(x,y,\phi) & =  {\cal O}(r^{-2})\\
     {\cal M}_3f(x,y,\phi)  & =  {\cal O}(r^{-10/3})\\
      {\cal M}_4f(x,y,\phi) & =  {\cal O}(r^{-14/3})\\
      {\cal N}_5f(x,y,\phi) & =  {\cal O}(r^{1/3})\\
     {\cal N}_6f(x,y,\phi) & =  {\cal O}(r^{-1})\\
      {\cal N}_7f(x,y,\phi) & =  {\cal O}(r^{-7/3}) 
\end{split}
 \end{equation} 
as $r\to 0$, and all these functions have at most polynomial growth as $(x,y) \to \infty$.

\s

From the formulas (\ref{G1}, \ref{G2}, \ref{Gamma}) for $G_1(h)$, $G_2(h)$ and $\Gamma(h)$ and equation \eqref{G bis} we have
  $$
 \int_{H^{-1}(h)} a(x,y) \,dm_h(x,y) =\int_{H^{-1}(h)} b(x,y) \,dm_h(x,y) =  \big|G_1(h)\big|^{2/3} \bigl(G_2(h)\bigr)^{1/3} = G(h)
 $$
 for $h > 0$.
It follows that there exist functions  $A(x,y)$ and
$B(x,y)$ defined for $(x,y) \neq (0,0)$ such that
 \begin{equation} \label{AB}
 \begin{split}
 (U_1.A)(x,y)  & =  a(x,y) - G(H(x,y)) \\
 (U_1.B)(x,y)  & =  b(x,y) - G(H(x,y)).
 \end{split} \end{equation}
On each Hamiltonian orbit $H^{-1}(h) = \{(x,y): H(x,y) = h\}$ the functions $A$ and $B$ are well-defined up to additive constants, so we may assume $A(0,y) = B(0,y)$ for all $y > 0$.    
Next define the operator
  \begin{equation}
  \overline{\cal M} =\left(-\sin^2 \phi -\cos^3 \phi  \sin \phi \right)
      \frac{\partial}{\partial \phi}
      + \frac{1}{2} \cos^4 \phi \frac{\partial^2}{\partial
         \phi^2}. \nonumber
 \end{equation}
It is hypoelliptic on $\R/(2 \pi \Z)$, with invariant probability measure $\nu$, say, on $\R/(2 \pi \Z)$. Define
 \begin{equation}
 \gamma_0 = \int_{\R/(2 \pi \Z)} \left( \cos \phi \sin \phi +
             \frac{1}{2}\cos ^2 \phi - \cos^2 \phi \sin^2
             \phi\right) \,d\nu(\phi), \nonumber 
 \end{equation}
then there is a smooth bounded function $R(\phi)$ so that
 \begin{equation}\label{step2}
 \overline{\cal M}R(\phi) =  \cos \phi \sin \phi +
             \frac{1}{2}\cos ^2 \phi - \cos^2 \phi \sin^2
             \phi - \gamma_0.
  \end{equation}
   Here $\gamma_0$ is a fixed constant.  As in \cite{BG02} we can
identify it as the top Lyapunov exponent of the linear SDE
 $$
 dv_t = \left[\begin{array}
          {cc} 0 & 1\\0 & 0 \end{array}\right]v_t dt +
     \left[\begin{array}
               {cc} 0 & 0\\1 & 0 \end{array}\right]v_t dW_t.
 $$
Numerically $\gamma_0 \sim 0.29$.  The exact formula for $\gamma_0$ given in Theorem \ref{thm main} is due to Ariaratnam and Xie \cite{AX90}.

Now define $F_\varepsilon(x,y,\phi)$ by
 \begin{align}
 F_\varepsilon(x,y,\phi) & = 
     A(x,y)\Bigl( \cos \phi \sin \phi - \sin^2 \phi \, R'(\phi) \Bigr) \nonumber \\
     & \quad  + B(x,y)\left(\frac{1}{2}\cos ^2 \phi - \cos^2 \phi \sin^2 \phi + \cos^3 \phi \sin \phi \, R'(\phi)-\frac{1}{2}\cos ^4 \phi \, R''(\phi)\right) \nonumber\\
    & \quad +  \varepsilon^{-2/3}R(\phi) \nonumber\\
    & =:  A(x,y)\Psi_1(\phi) +B(x,y)\Psi_2(\phi) + \e^{-2/3}R(\phi),  \label{Feps}  
\end{align}
say.  Using \eqref{Q1}, \eqref{AB} and \eqref{step2} we have 
  \begin{align}
 (U_1+\e^{2/3} {\cal M}_0)F_\varepsilon(x,y,\phi)
 &  = Q_0(x,y,\phi) - \gamma_0
  G(H(x,y))  \nonumber \\
  & \quad  + \varepsilon^{2/3}\Bigl(A(x,y){\cal M}_0\Psi_1(\phi)+B(x,y){\cal M}_0\Psi_2(\phi)\Bigr).  \label{L0Feps}
  \end{align}
Using \eqref{Leps} and \eqref{L0Feps}, 
\begin{align}
   {\cal L}_\e F_\e(x,y,\phi) & = Q_0(x,y,\phi) - \gamma_0
  G(h) \nonumber \\
  & \quad + \e^2\overline{\cal L}A(x,y) \Psi_1(\phi)+  \e^2\overline{\cal L}B(x,y) \Psi_2(\phi) \nonumber \\
  & \quad +A(x,y)\Bigl( \e^{2/3}{\cal M}_0 + \e^{4/3}{\cal M}_1+\e^2{\cal M}_2+\e^{8/3}{\cal M}_3+ \e^{10/3}{\cal M}_4\Bigr)\Psi_1(\phi) \nonumber \\
    & \quad +B(x,y)\Bigl( \e^{2/3}{\cal M}_0 + \e^{4/3}{\cal M}_1+\e^2{\cal M}_2+\e^{8/3}{\cal M}_3+ \e^{10/3}{\cal M}_4\Bigr)\Psi_2(\phi) \nonumber \\
    & \quad + (V_1.A)(x,y)\Bigl(\e^{4/3}{\cal N}_5+ \e^2{\cal N}_6+\e^{8/3}{\cal N}_7\Bigr)\Psi_1(\phi) \nonumber \\       
     & \quad + (V_1.B)(x,y)\Bigl(\e^{4/3}{\cal N}_5+ \e^2{\cal N}_6+\e^{8/3}{\cal N}_7\Bigr)\Psi_2(\phi) \nonumber\\
     & \quad + \e^{-2/3}\Bigl( \e^{4/3}{\cal M}_1+\e^2{\cal M}_2+\e^{8/3}{\cal M}_3+ \e^{10/3}{\cal M}_4\Bigr)R(\phi) . \label{LFeps}
  \end{align}
Note that there is no ${\cal M}_0$ term in the last line. 

\subsection{Estimates on the functions $A$ and $B$} \label{sec AB}

Up to this point explicit formulas have been available for all the functions involved in the calculations.  In this section we follow the methods established in \cite[Section 3.3]{BG02} to obtain some {\it a priori} estimates on the functions $A$ and $B$ defined in \eqref{AB}.    The first result is quoted unchanged from \cite{BG02}. 
    
\begin{lemma} \label{lem AB} {\rm \cite[Lemma 11]{BG02}}
 Let $f$ be a smooth
function on ${\bf R}^2 \setminus \{(0,0)\}$. Write $\tilde{f}(x,y) = f(x,y) - \int f \,dm_h$ for $(x,y)
\in H^{-1}(h)$.  Then there is a unique function $F$ on ${\bf R}^2
\setminus \{(0,0)\}$ such that $U_1.F(x,y) = \tilde{f}(x,y)$ and
$F(0,y) = 0$ for all $y > 0$.  Moreover we have the equality
  \begin{equation} \nonumber
  U_1.(U_2.F)(x,y)  = U_2.\tilde{f}(x,y) - J(x,y)\tilde{f}(x,y)
    \end{equation}
and the inequalities
   \begin{align*}
   |F(x,y)| & \le   T(h)\sup\{|\tilde{f}(u,v)|: (u,v) \in H^{-1}(h)\}, \\
   |(U_2.F)(x,y)| & \le   T(h)\sup\{|(U_2-J).\tilde{f}(u,v)|: (u,v) \in H^{-1}(h)\},
       \\
   |(U_2^2.F)(x,y)| & \le  T(h)\sup\{|(U_2-J)^2.\tilde{f}(u,v)|: (u,v) \in
   H^{-1}(h)\} 
   \end{align*}
valid for $(x,y) \in H^{-1}(h)$.
\end{lemma}
  
The change in $V_1$ from multiplicative to additive noise causes changes in \cite[Lemma 12]{BG02}.   Here is the new version. 
   
   \begin{lemma} \label{lem AB2}
For the system \eqref{add eps}, suppose $U_1.F(x,y) =
\tilde{f}(x,y)$ as above, and that $F(0,y)= 0$ for all $y > 0$.  For ease of
notation write $\|\tilde{f}\|_h =
\sup\{|\tilde{f}(u,v)|: (u,v) \in H^{-1}(h)\}$.

(i) There are constants $k$ and $\delta$ such that if $r =
\sqrt{x^2+y^2} < \delta$ and $H(x,y) = h$ then
  \begin{align*}
  |F(x,y)| & \le  k \|\tilde{f}\|_h \\
|V_1.F(x,y)| & \le  k\left[ r^{-1}\|\tilde{f}\|_h+ r\|U_2.\tilde{f}\|_h \right] \\
  \left|\overline{\cal L}F(x,y)\right|
   & \le   \left[r^{-2}\|\tilde{f}\|_h + r^{-1}\|V_1.\tilde{f}\|_h +\|U_2.\tilde{f}\|_h + r^2\|U_2^2.\tilde{f}\|_h \right]
     \end{align*}

(ii) If $\tilde{f}$ and all of its first and second partial derivatives have at most polynomial
growth as $(x,y) \to \infty$, then so do $F$ and $V_1.F$ and
$\overline{\cal L}F$.
\end{lemma}

\proof We convert estimates on $U_2.F$ and $U_2^2.F$ into estimates on $V_1.F$ and $\overline{\cal L}F$.   Since $V_i(x,y) = \alpha_i^1(x,y) U_1(x,y)+ \alpha_i^2(x,y) U_2(x,y)$ we have 
   \begin{align*}
   (V_i.F)(x,y) & = \alpha_i^1(x,y) (U_1.F)(x,y) + \alpha_i^2(x,y) (U_2.F)(x,y)\\
   & =  \alpha_i^1(x,y) \tilde{f}(x,y)
   + \alpha_i^2(x,y) (U_2.F)(x,y)
   \end{align*} 
and 
   \begin{align*}
   \bigl(V_1.(V_1.F)\bigr)(x,y) & = (V_1.\alpha_1^1)(x,y) \tilde{f}(x,y) + \alpha_1^1(x,y)V_1.\tilde{f}(x,y) + (V_1.\alpha_1^2)(x,y) (U_2.F)(x,y) \\
     & \quad  + \alpha_1^2(x,y)\Bigl( \alpha_1^1(x,y) \bigl(U_1.(U_2.F)\bigr)(x,y) + \alpha_1^2(x,y)(U_2^2.F)(x,y)\Bigr).
  \end{align*}
Note also that
      $$\|(U_2-J).\tilde{f}\|_h  \le \|U_2.\tilde{f}\|_h+  \|J\|_h \|\tilde{f}\|_h
      $$
and 
  $$\|(U_2-J)^2.\tilde{f}\|_h \le \|U_2^2.\tilde{f}\|_h + 2\|J\|_h \| \|U_2.\tilde{f}\|_h+ \bigl(\|U_2.J\|_h+\|J\|_h^2\bigr)\|\tilde{f}\|_h.
$$
The results now follow from Lemma \ref{lem AB}.  In (i) the first inequality uses the fact that $T(h)$ is bounded as $h \to 0$.  The second inequality uses the additional facts that $\alpha_1^1(x,y) = {\cal O}(r^{-1})$ and $\alpha_1^2 = {\cal O}(r)$ and $J(x,y) = {\cal O}(1)$.  The third inequality uses also the facts that $\alpha_0^1(x,y) = {\cal O}(1)$ and $\alpha_0^2(x,y)= {\cal O}(r^2)$ and $V_1.\alpha_1^1(x,y) = {\cal O}(r^{-2})$ and $V_1.\alpha_1^2(x,y) = {\cal O}(1)$ and $U_2.J(x,y) = {\cal O}(1)$.  In (ii) we
use the fact that all of the functions just mentioned have at most
polynomial growth. 
\endproof

In place of \cite[Corollary 5]{BG02} we have 

     \begin{lemma} \label{lem AB3}
For the functions $A$ and $B$ satisfying \eqref{AB} with $A(0,y) = B(0,y)$ for all $y >0$, 
    \begin{equation} \label{AB growth}
 \begin{split}
       A(x,y) \mbox{ and } B(x,y) & \mbox{ are }{\cal O}(r^{2/3})\\
       V_1.A(x,y) \mbox{ and } V_1.B(x,y) & \mbox{ are } {\cal O}(r^{-1/3}) \\
       \overline{\cal L}A(x,y) \mbox{ and }\overline{\cal L}B(x,y) & \mbox{ are }{\cal O}(r^{-4})
   \end{split}
   \end{equation}
   as $r \to 0$, and $A$ and $B$ and $V_1.A$ and $V_1.B$ and $\overline{\cal L}A$ and $\overline{\cal L}B$ all have at most polynomial growth as $r \to \infty$.
\end{lemma}

\proof  Write $\tilde{a}(x,y) = a(x,y) - G(H(x,y))$ and $\tilde{b}(x,y) =   b(x,y) - G(H(x,y))$.  Using the formulas \eqref{ab} for $a(x,y)$ and $b(x,y)$ together with the estimates for $\Gamma(h)$ and $G(h)$ in Lemma \ref{lem GG} we get 
 \begin{align*}
       \tilde{a}(x,y) \mbox{ and }\tilde{b}(x,y) & \mbox{ are }  {\cal O}(r^{2/3})\\
       U_2.\tilde{a}(x,y) \mbox{ and }U_2.\tilde{b}(x,y) & \mbox{ are } {\cal O}(r^{-4/3}) \\
       V_1.\tilde{a}(x,y) \mbox{ and }V_1.\tilde{b}(x,y) & \mbox{ are } {\cal O}(r^{-1/3}) \\
       U_2^2.\tilde{a}(x,y) \mbox{ and }U_2^2.\tilde{b}(x,y) & \mbox{ are } {\cal O}(r^{-10/3})
   \end{align*}
as $r \to 0$, and that all eight functions have at most polynomial growth.  The results now follow immediately from Lemma \ref{lem AB2}.  \endproof

\subsection{Replacement of $Q_0(x,y,\phi)$ with $\gamma_0 G(h)$}

 \begin{theorem} \label{thm QG} For the system \eqref{add eps} there exists $K_2$ such that such that for all $0 < \e \le 1$ and $(x_0,y_0,v_0)$ with $v_0 \neq 0$   
  $$
  \limsup_{n \to \infty} \frac{1}{n} \left|\int_0^n \chi_\e(h_s) \Bigl( Q_0(x_s,y_s,\phi_s) - \gamma_0 G(h_s)\Bigr)ds \right|  \le K_2 \e^{2/3}
  $$
almost surely.
 \end{theorem}   

\proof
Recall the function $F_\e(x,y,\phi)$ in \eqref{Feps} and the formulas for ${\cal L}_\e F_\e(x,y,\phi)$ and $d \chi_\e(h_t)$ in \eqref{LFeps} and  \eqref{dchi}.  For brevity of notation write $(x_t,y_t,\phi_t) = z_t$.
Then
  \begin{align}
  d\bigl(\chi_\e(h_t)F_\e(z_t)\bigr) & = 
  \chi_\e(h_t) d F_\e(z_t) + F_\e(z_t) d\chi_\e(h_t)+ d\chi_\e(h_t)dF_\e(z_t)  \nonumber\\
    & = \chi_\e(h_t){\cal L}_\e F_\e(z_t)dt \nonumber \\
     & \quad  + \chi_\e(h_t)\Bigl( \e V_1.F_\e(z_t) + (\e^{1/3}{\cal N}_5+\e{\cal N}_6+\e^{5/3}{\cal N}_7)F_\e(z_t)\Bigr)dW_t \nonumber\\
     & \quad + F_\e(z_t) \Bigl(\e^2\chi_\e'(h_t) \overline{\cal L}H(x_t,y_t) + \frac{\e^2}{2} \chi_\e''(h_t)\bigl((V_1.H)(x_t,y_t)\bigr)^2\Bigr)dt \nonumber\\
       & \quad + F_\e(z_t) \e \chi_\e'(h_t)V_1.H(x_t,y_t)dW_t \nonumber\\
       & \quad + \e \chi_\e'(h_t)V_1.H(x_t,y_t) \Bigl( \e V_1.F_\e(z_t) + (\e^{1/3}{\cal N}_5+\e{\cal N}_6+\e^{5/3}{\cal N}_7)F_\e(z_t)\Bigr)dt  \nonumber \\
       & =\chi_\e(h_t)\bigl( Q_0(x,y,\phi) - \gamma_0  G(h)\bigr)dt \nonumber \\
  & \quad + \chi_\e(h_t)\Bigl(\e^2\overline{\cal L}A(x,y) \Psi_1(\phi)+  \e^2\overline{\cal L}B(x,y) \Psi_2(\phi)\Bigr)dt \nonumber \\
  & \quad + \chi_\e(h_t)A(x,y)\Bigl( \e^{2/3}{\cal M}_0 + \e^{4/3}{\cal M}_1+\e^2{\cal M}_2+\e^{8/3}{\cal M}_3+ \e^{10/3}{\cal M}_4\Bigr)\Psi_1(\phi)dt \nonumber \\
    & \quad + \chi_\e(h_t)B(x,y)\Bigl( \e^{2/3}{\cal M}_0 + \e^{4/3}{\cal M}_1+\e^2{\cal M}_2+\e^{8/3}{\cal M}_3+ \e^{10/3}{\cal M}_4\Bigr)\Psi_2(\phi)dt \nonumber \\
    & \quad + \chi_\e(h_t)(V_1.A)(x,y)\Bigl(\e^{4/3}{\cal N}_5+ \e^2{\cal N}_6+\e^{8/3}{\cal N}_7\Bigr)\Psi_1(\phi) dt\nonumber \\       
     & \quad + \chi_\e(h_t)(V_1.B)(x,y)\Bigl(\e^{4/3}{\cal N}_5+ \e^2{\cal N}_6+\e^{8/3}{\cal N}_7\Bigr)\Psi_2(\phi) dt\nonumber\\
     & \quad + \chi_\e(h_t)\e^{-2/3}\Bigl( \e^{4/3}{\cal M}_1+\e^2{\cal M}_2+\e^{8/3}{\cal M}_3+ \e^{10/3}{\cal M}_4\Bigr)R(\phi)dt \nonumber \\
     & \quad  + \chi_\e(h_t)\Bigl( \e V_1.F_\e(z_t) + (\e^{1/3}{\cal N}_5+\e{\cal N}_6+\e^{5/3}{\cal N}_7)F_\e(z_t)\Bigr)dW_t \nonumber\\
     & \quad + F_\e(z_t) \Bigl(\e^2\chi_\e'(h_t) \overline{\cal L}H(x_t,y_t) + \frac{\e^2}{2} \chi_\e''(h_t)\bigl((V_1.H)(x_t,y_t)\bigr)^2\Bigr)dt \nonumber\\
       & \quad + F_\e(z_t) \e \chi_\e'(h_t)V_1.H(x_t,y_t)dW_t \nonumber\\
       & \quad + \e \chi_\e'(h_t)V_1.H(x_t,y_t) \Bigl( \e V_1.F_\e(z_t) + (\e^{1/3}{\cal N}_5+\e{\cal N}_6+\e^{5/3}{\cal N}_7)F_\e(z_t)\Bigr)dt.  \label{dFeps}    
  \end{align}   
Now integrate from $0$ to $n$ and divide by $n$.  On the left by Lemma \ref{lem AB3} the function $(x,y) \mapsto \sup\{\chi_\e(h)|F_\e(x,y,\phi)|: \phi \in \R/(2 \pi \Z)\}$ is $\mu$ integrable, so by Lemma \ref{lem ergod}(i) we have 
      $$
      \frac{1}{n} \big|\chi_\e(h_n)F_\e(z_n) - \chi_\e(h_0)F_\e(z_0) \big| \to 0
      $$
almost surely.  On the right we keep the term $\frac{1}{n} \left(\int_0^n \chi_\e(h_s) \bigl( Q_0(x_s,y_s,\phi_s) - \gamma_0 G(h_s)\bigr)ds \right)$.  It remains to obtain bounds as $\e \to 0$ on each of the remaining terms on the right using the methods seen earlier in the proof of Theorem \ref{thm lamQ}.  We will use Lemmas \ref{lem ergod} and \ref{lem gint} frequently.  In this calculation we will distinguish estimates as $r \to 0$ and estimates as $\e \to 0$ by writing ${\cal Q}(r^k)$ to denote a function $f(x,y,\phi)$ such that $\limsup_{r \to 0} |f(x,y,\phi)|/r^k <\infty$ and ${\cal O}(\e^k)$ to denote a function $g(\e)$ such that $\limsup_{\e \to 0} |g(\e)|/\e^k <\infty$.   So for example $\overline{\cal L}H(x,y,\phi) = -\beta y^2 + \sigma^2/2 = {\cal Q}(1)$ and $V_1.H(x,y,\phi) = \sigma y = {\cal Q}(r)$ and then
  $${\cal L}_\e(\chi_\e \circ H) = \e^2\chi_\e'(h) \overline{\cal L}H + \frac{\e^2}{2} \chi_\e''(h)[(V_1.H)]^2 = 
  \e {\cal Q}(1) + {\cal Q}(r^2).
  $$

(i) Terms involving $\overline{\cal L}A$ and $A$ and $V_1.A$ in \eqref{dFeps}. Before multiplying by the factor $\chi_\e(h_t)$, and using \eqref{AB growth} and \eqref{MN growth}, these terms have magnitude   
 \begin{multline}
  \e^2{\cal Q}( r^{-4}) + {\cal Q}(r^{2/3})\Bigl(  \e^{2/3}{\cal Q}(r^{2/3}) + \e^{4/3}{\cal Q}(r^{-2/3}) +\e^2{\cal Q}(r^{-2})+\e^{8/3}{\cal Q}(r^{-10/3})+ \e^{10/3}{\cal Q}(r^{-14/3})\Bigr)\\
   +  {\cal Q}(r^{-1/3})\Bigl(\e^{4/3}{\cal Q}(r^{1/3})+ \e^2{\cal Q}(r^{-1})+\e^{8/3}{\cal Q}(r^{-7/3})\Bigr) \nonumber
  \end{multline}
as $r \to 0$.   Now multiply by $\chi_\e(h)$ and consider $\limsup_{n \to \infty} \frac{1}{n} \left|\int_0^n \cdot ds\right|$ for each of the terms.  Using Lemmas \ref{lem ergod} and \ref{lem gint} we get
 \begin{multline}
  \e^2{\cal O}(\e^{-1}) +\Bigl(  \e^{2/3}{\cal O}(1) + \e^{4/3}{\cal O}(1) +\e^2{\cal O}(1)+\e^{8/3}{\cal O}(\e^{-1/3})+ \e^{10/3}{\cal O}(\e^{-1})\Bigr)\\
   +  \Bigl(\e^{4/3}{\cal O}(1)+ \e^2{\cal O}(1)+\e^{8/3}{\cal O}(\e^{-1/3})\Bigr). \nonumber
  \end{multline}
Together these terms are ${\cal O}(\e^{2/3})$.  Similarly for the terms involving $\overline{\cal L}B$ and $B$ and $V_1.B$. 

(ii) Terms involving ${\cal M}_iR$ for $1 \le i \le 4$ in \eqref{dFeps}.  Before multiplying by the factor $\chi_\e(h_t)$, and using \eqref{MN growth}, these terms have magnitude    
   $$
    \e^{-2/3}\Bigl( \e^{4/3}{\cal Q}(r^{-2/3})+\e^2{\cal Q}(r^{-2})+\e^{8/3}{\cal Q}(r^{-10/3})+ \e^{10/3}{\cal Q}(r^{-14/3})\Bigr)
 $$
as $r \to \infty$.  Now multiply by $\chi_\e(h)$ and consider $\limsup_{n \to \infty} \frac{1}{n} \left|\int_0^n \cdot ds\right|$ for each of the terms.  Using Lemmas \ref{lem ergod} and \ref{lem gint} we get
   $$
    \e^{2/3}{\cal O}(1)+\e^{4/3}{\cal O}(|\log \e|)+\e^2{\cal O}(\e^{-2/3})+ \e^{8/3}{\cal Q}(\e^{-4/3}) = {\cal O}(\e^{2/3}).
 $$
 
(iii)  The $dt$ terms involving $F_\e$ (but not its derivatives).  Using the bounds  $|\chi_\e'(h)| \le c_1/\e$ and $|\chi_\e''(h)| \le c_2/\e^2$ together with \eqref{AB growth}, these terms have magnitude 
  \begin{align*}
  \Bigl({\cal Q}(r^{2/3}) + \e^{-2/3} {\cal Q}(1)\Bigr)\Bigl( \e{\cal Q}(1) + {\cal Q}(r^2)\Bigr)
  = \e^{-2/3} {\cal Q}(r^2) + {\cal Q}(r^{8/3})+ \e^{1/3} {\cal Q}(1)+ \e {\cal Q}(r^{2/3}) 
 \end{align*} 
as $r \to 0$.  Now consider $\limsup_{n \to \infty} \frac{1}{n} \left|\int_0^n \cdot ds\right|$ for each of the terms, and note that the integrands are all 0 outside the set $\{(x,y): \e \le h(x,y) \le 2\e\}$.  Using Lemmas \ref{lem ergod} and \ref{lem gint} we get 
  $$ \e^{-2/3} {\cal O}(\e^2) + {\cal O}(\e^{4/3}) + \e^{1/3} {\cal O}(\e)+ \e {\cal O}(\e^{4/3}) = {\cal O}(\e^{4/3}).
  $$
 
(iv)  The $dt$ terms involving the derivatives $V_1.F_\e$ and ${\cal N}_i F_\e$ for $5 \le i \le 7$.  Using the bound  $|\chi_\e'(h)| \le c_1/\e$ together with \eqref{MN growth} and \eqref{AB growth}, these terms have magnitude 
  \begin{align*}
 \lefteqn{{\cal Q}(r) \Bigl( \e {\cal Q}(r^{-1/3}) + \e^{-2/3} \bigl(\e^{1/3}{\cal Q}(r^{1/3})+\e{\cal Q}(r^{-1})+\e^{5/3}{\cal Q}(r^{-7/3})\Bigr)} \hspace{20ex}  \\
  &=  \e{\cal Q}(r^{2/3}) + \e^{-1/3} {\cal Q}(r^{4/3}) + \e^{1/3}{\cal Q}(1) + \e {\cal Q}(r^{-4/3})
  \end{align*} 
as $r \to 0$.   Now consider $\limsup_{n \to \infty} \frac{1}{n} \left|\int_0^n \cdot ds\right|$ for each of the terms, and note that the integrands are all 0 outside the set $\{(x,y): \e \le h(x,y) \le 2\e\}$.  Using Lemmas \ref{lem ergod} and \ref{lem gint} we get 
  $$\e{\cal O}(\e^{4/3}) + \e^{-1/3} {\cal O}(\e^{5/3}) + \e^{1/3}{\cal O}(\e) + \e {\cal O}(\e^{1/3}) = {\cal O}(\e^{4/3}).
  $$
  
(v)  The integrands in each of the $dW_t$ terms are bounded by $L^2(\mu)$ functions of $(x,y)$, so by Lemma \ref{lem ergod} they satisfy $\lim_{n \to \infty} \frac{1}{n} \int_0^n \cdot \,dW_s= 0 $ almost surely.

Together, all the $dW_t$ integrals have $\lim_{t \to \infty} \frac{1}{t}\int_0^t \cdots dW_s = 0$ almost surely and all the $dt$ integrals except the one with $\chi_\e(h_t)(Q_0(x_t,y_t,\phi_t)-\gamma G(h_t))$ have bounds on the almost sure value of the $\limsup_{n \to \infty} \frac{1}{n}\left|\int_0^t \cdots ds\right|$ and the bounds are non-random and do not depend on the initial value $(x_0,y_0,v_0)$ with $v_0 \neq 0$.  As $\e \to 0$ the biggest of these bounds is ${\cal O}(\e^{2/3})$, and this completes the proof of Theorem \ref{thm QG} \endproof

\section{Proof of Theorem \ref{thm main}} \label{sec proof}
  
\begin{proposition} \label{prop Geps}  For the system \eqref{add eps} there exists $K_3$ such that for all $(x_0,y_0)$  
  $$
  \limsup_{t \to \infty} \left| \frac{1}{t}\int_0^t \chi_\e(h_s) G(h_s)ds -  \int_{\R^2} G(h)d\mu(x,y)  \right|  \le K_3 \e^{4/3}
  $$
almost surely.
 \end{proposition}   

\proof  By Lemma \ref{lem GG} we have $G(H(x,y)) \in L^1(\mu)$.  The almost sure ergodic theorem implies that for all $(x_0,y_0)$ we have
   $$
   \frac{1}{t}\int_0^t \chi_\e(h_s) G(h_s)ds = \int _{\R^2} \chi_\e(h) G(h) d\mu(x,y)
   $$
almost surely.  Then
    \begin{align*}
  \limsup_{t \to \infty} \left| \frac{1}{t}\int_0^t \chi_\e(h_s) G(h_s)ds -  \int_{\R^2} G(h)d\mu(x,y)  \right|  &  = \left| \int_{\R^2}(1-\chi_\e(h))G(h)d\mu(x,y) \right| \\
  & \le \int_{h \le 2 \e} G(h)d\mu(x,y) \\
  & = {\cal O}(\e^{4/3})
  \end{align*}
as $\e \to 0$, where the last estimate uses $G(h) = {\cal O}(h^{1/3}) = {\cal O}(r^{2/3})$ as $r \to 0$.  \endproof

\n {\it Proof of Theorem \ref{thm main}}  Let $0 < \e \le 1$.  For any $(x_0,y_0,v_0)$ with $v_0 \neq 0$ we have
   \begin{align*}
   \left| \lambda(\e) - \e^{2/3} \gamma_0 \int_{\R^2} G(h)d\mu(x,y) \right| & 
    \le \left|\lambda(\e) - \frac{1}{n} \log \|v_n\|\right|+ \frac{1}{n}\Big|\log\|v_n\| - \Lambda_\e(n) \Big| \\
    & \quad + \frac{1}{n} \left| \Lambda_\e(n) - \e^{2/3} \int_0^n \chi_\e(h_s)Q_0(x_s,y_s,\phi_s)ds \right|\\
    & \quad + \frac{\e^{2/3}}{n} \left|\int_0^n \chi_\e(h_s)\bigl(G(h_s)-Q(x_s,y_s,\phi_s)\bigr)ds \right|\\
    & \quad +  \e^{2/3} \left| \frac{1}{n}\int_0^n \chi_\e(h_s) G(h_s)ds -  \int_{\R^2} G(h)d\mu(x,y)  \right| \end{align*}
for all $n \ge 1$.  Taking the almost sure $\limsup$ of the right side as $n \to \infty$, and using Propositions \ref{prop FK cart} and \ref{prop equiv}, Theorems \ref{thm lamQ} and \ref{thm QG}, and finally Proposition \ref{prop Geps} we obtain
   $$
   \left| \lambda(\e) - \e^{2/3} \gamma_0 \int_{\R^2} G(h)d\mu(x,y) \right|
    \le 0+ 0 + K_1 \e^{4/3}+ K_2 \e^{4/3} + K_3 \e^{2} \le (K_1+K_2+K_3)\e^{4/3},
    $$
as required.  \endproof

\appendix

 \section{Appendix}\label{sec appendix}

The results in the appendix apply to the system \eqref{nonlin} with a general $C^\infty$ potential ${\cal U}(x)$.  Note that the conditions imposed on ${\cal U}$ in Lemmas \ref{lem Hor} and \ref{lem cont} are satisfied when ${\cal U}(x) = ax^2/2+bx^4/4$ with $b > 0$, and so these Lemmas can be used in the proof of Proposition \ref{prop FK cart}.  
The system \eqref{nonlin} can be written
  \begin{align*}
   dx_t & = y_tdt \\
  dy_t & = (-{\cal U}'(x_t)- \beta y_t)dt+ \sigma \circ dW_t
   \end{align*}
with parameters $\beta$ and $\sigma \neq 0$.
Linearizing gives
   \begin{equation} \label{v}
   dv_t = \begin{bmatrix} 0 & 1 \\ -{\cal U}''(x_t) & -\beta \end{bmatrix} v_t dt.
  \end{equation}
Putting $v = \|v\| \begin{bmatrix} \cos \theta \\ \sin \theta \end{bmatrix}$ we get 
 \begin{equation} \nonumber 
   d\theta_t  = \Bigl(-\sin^2 \theta_t - {\cal U}''(x_t) \cos^2 \theta_t - \beta \sin\theta_t \cos \theta_t \Bigr)dt .  
   \end{equation}
The process $\{(x_t,y_t,\theta_t): t \ge 0\}$ satisfies an SDE of the form 
    \begin{equation} \label{z}
    dz_t = X_0(z_t)dt+X_1(z_t)\circ dW_t
     \end{equation}
 with
      $$
  X_0(x,y,\theta) = \begin{bmatrix} y \\ -{\cal U}'(x) -\beta y \\-\sin^2 \theta - {\cal U}''(x) \cos^2 \theta - \beta \sin\theta \cos \theta \end{bmatrix}  \quad \mbox{ and } \quad  X_1(x,y,\theta) = \begin{bmatrix} 0 \\ \sigma \\ 0 \end{bmatrix}.
  $$ 

\subsection{Parabolic H\"{o}rmander condition}

Let $L = L(X_0, X_1)$ be the Lie algebra generated by the vector fields $X_0$ and $X_1$ and let $L_0$ be the ideal in $L$ generated by $X_1$.  We say the system satisfies the parabolic H\"{o}rmander condition, see for example Hairer \cite{Hai}, if 
   $$
    \mbox{span}\{Y(x,y,\theta): Y \in L_0\} = \R^3 \quad \mbox{ for all }(x,y,\theta) \in \R^2 \times \R/(2 \pi \Z)
    $$

\begin{lemma} \label{lem Hor} Assume $\sigma \neq 0$ and ${\cal U}^{(4)}(x) \neq 0$ for all $x \in \R$.  Then the  system \eqref{z} satisfies the parabolic H\"{o}rmander condition. 
\end{lemma}

\proof  Without loss of generality we may assume $\sigma = 1$.  Instead of looking at the vector fields $X_0$ and $X_1$  for the sde for $(x,y,\theta)$, it is easier to calculate the Lie brackets for the vector fields for $(x,y,v)\in \R\times\R \times \R^2$ given by 
$$
  \widehat{X}_0(x,y,v) = \begin{bmatrix} y \\ -{\cal U}'(x) -\beta y \\(C-{\cal U}''(x)D)v \end{bmatrix}  \quad \mbox{ and } \quad  \widehat{X}_1(x,y,v) = \begin{bmatrix} 0 \\ 1 \\ 0 \end{bmatrix},
  $$ 
where 
  $$
  C = \begin{bmatrix} 0 & 1 \\0 & -\beta \end{bmatrix} \quad \mbox{ and } \quad D = \begin{bmatrix} 0 & 0 \\1 & 0 \end{bmatrix}.
  $$
Note that
  $$
  [C,D]= \begin{bmatrix} 1 & 0 \\ -\beta & -1 \end{bmatrix} \quad \mbox{and} \quad \bigl[C,[C,D]\bigr] = \begin{bmatrix} -\beta & -2 \\ \beta^2 & \beta \end{bmatrix} \quad \mbox{and} \quad \bigl[D,[C,D]\bigr] = 2D.
  $$
It follows that  
   $$
   \begin{bmatrix} 0 & 1 \\-1 & 0 \end{bmatrix} \in \mbox{span}\Big\{ D, [C,D], \bigl[ C,[C,D]\bigr]\Big\}.
   $$
 Direct calculation gives
    $$
   \widehat{X}_2(x,y,v):= [\widehat{X}_1,\widehat{X}_0](x,y,v) + \beta \widehat{X}_1(x,y,v) = \begin{bmatrix} 1 \\ 0 \\ 0 \end{bmatrix}
    $$
and 
       $$
  \widehat{X}_3(x,t,v):=\Bigl[\widehat{X}_2,[\widehat{X}_2,\widehat{X}_0]\Bigr](x,y,v) = \begin{bmatrix} 0 \\-{\cal U}^{(3)}(x) \\ -{\cal U}^{(4)}(x) Dv \end{bmatrix}
   $$
and 
  \begin{align*}
\widehat{X}_4(x,y,v)& := [\widehat{X}_3,\widehat{X}_0](x,y,v) = \begin{bmatrix} -{\cal U}^{(3)}(x) \\[1ex] -\beta {\cal U}^{(3)}(x) + {\cal U}^{(4)}(x)y  \\[1ex] -{\cal U}^{(4)}(x)[C,D]v - {\cal U}^{(5)}(x)y Dv \end{bmatrix}.
    \end{align*}  
Finally the $v$ component of      
    \begin{align*}
\widehat{X}_5(x,y,v) := [\widehat{X}_4,\widehat{X}_0](x,y,v)
        \end{align*}
 is 
  \begin{align*}
  \lefteqn{ -{\cal U}^{(4)}(x)[C,[C,D]]v } \\
  & \quad + \Bigl(\bigl({\cal U}^{(3)}(x)\bigr)^2+{\cal U}^{(6)}(x)y - {\cal U}'(x){\cal U}^{(5)}(x)- \beta {\cal U}^{(5)}(x)y + 2{\cal U}''(x){\cal U}^{(4)}(x) \Bigr)Dv.
  \end{align*}
In the last line we use $\bigl[D,[C,D]\bigr] = 2D$.  It follows that for any given $(x,y,v)$ with ${\cal U}^{(4)}(x) \neq 0$ the span of the $v$ components of $\{\widehat{X}_j(x,y,v): j=3,4,5\}$ contains the vector $\begin{bmatrix} 0 & 1 \\-1 & 0 \end{bmatrix}v$.  Hence for any given $(x,y,\theta)$ with ${\cal U}^{(4)}(x) \neq 0$ the span of the $\theta$ components of $\{X_j(x,y,\theta): j=3,4,5\}$ contains the vector $1$.  Since the span of $X_1(x,y,\theta)$ and $X_2(x,y,\theta)$ is $\R^2 \times \{0\}$ it follows that the span of $\{X_j(x,y,\theta): j=1,2,3,4,5\}$ is $\R^3$ for all $(x,y,\theta)$ with ${\cal U}^{(4)}(x) \neq 0$ .  \endproof

       \subsection{Controllability}    

\begin{lemma}   \label{lem cont}  Suppose ${\cal U}''(x)$ is unbounded above.  Given any initial condition $(x_0,y_0, \theta_0) \in \R^2 \times \R/(2 \pi Z)$, any target point $(x_1,y_1,\theta_1) \in \R^2 \times \R/(2 \pi Z)$ and any $\e > 0$ and any $T \ge 0$ there exists a piecewise continuous $h: [0,T] \to \R$ such that the unique solution of 
     \begin{equation} \label{cont}
     (\dot{x}(t),\dot{y}(t),\dot{\theta}(t)) = X_0(x(t),y(t),\theta(t)) + h(t)X_1(x(t),y(t),\theta(t))
     \end{equation}
with initial condition $(x_0,y_0,\theta_0)$ satisfies $\|(x(T),y(T),\theta(T)) - (x_1,y_1,\theta_1)\| < \e$.
\end{lemma}

\proof The proof adapts methods used by Coti Zelati and Hairer \cite[Proposition 3.3]{CZH21}.  Since $y$ in \eqref{cont} can be completely controlled by $h$, the proof reduces to finding a
piecewise $C^1$ function $y:[0,T] \to \R$ with $y(0) = y_0$ and $y(T) = y_1$ so that the solution of the ODE
    \begin{align}  \label{ode}
     \begin{split}
       \dot{x}(t) &= y(t) \\
       \dot{\theta}(t) & = g(x(t),\theta(t))
       \end{split}  \end{align}
where  
  $$g(x,\theta) = -\sin^2 \theta - {\cal U}''(x) \cos^2 \theta - \beta \sin\theta \cos \theta
  $$ sends $(x_0,\theta_0)$ at time $t = 0$ to an $\e$-neighborhood of $(x_1,\theta_1)$ at time $T$.

The matrix in \eqref{v} has eigenvalues $-\beta/2\pm i \sqrt{{\cal U}''(x)- \beta^2/4}$ so long as ${\cal U}''(x) > \beta^2/4$.  Thus for fixed $x$, if ${\cal U}''(x) > \beta^2/4$ then $g(x,\theta) < 0$ for all $\theta$ and the ODE $\dot{\theta} = g(x,\theta)$ has period $2\pi/\sqrt{{\cal U}''(x)-\beta^2/4}$.  

Choose $\hat{x}$ so that the ODE $\dot{\theta} = g(\hat{x},\theta)$ sends $\theta_0$ to $\theta_1$ in exactly time $T$.  This is possible because ${\cal U}''(x) \to \infty$.  Let $\Phi(t,\theta)$ denote the time $t$ flow along $g(\hat{x},\cdot)$.  There exists $K_1$ (depending on $\hat{x}$) such that
   $$
   |\Phi(s,\theta)-\Phi(t,\phi)| \le K_1(|s-t| + |\theta-\phi|).
   $$
 
\begin{lemma} \label{lem path}  Given $(x_0,y_0)$ and $(x_1,y_1)$ for all $\delta > 0$ there exists a piecewise linear function $\{y(t): 0 \le t \le \delta\}$ with $y(0) = y_0$ and $y(\delta) = y_1$ such that $\int_0^\delta y(s)ds = x_1-x_0$.  Moreover the path $x(t) = x_0+\int_0^t y(s)ds$ remains bounded as $\delta \to 0$.
\end{lemma}

\proof Take $y(t)$ to be linear on $[0,\delta/2]$ and $[\delta/2,\delta]$ with $y(0) = y_0$ and $y(\delta/2) = a$ and $y(\delta) = y_1$.  Then we require
   $$
   x_1-x_0 = \int_0^\delta y(t)dt = \frac{\delta(y_0 +2a+y_1)}{4}
   $$
so that 
    $$
      a = \frac{2(x_1-x_0)}{\delta} -\frac{(y_0+y_1)}{2}.
      $$
Since $a$ grows at most like $1/\delta$ then $\sup\{|y(t)|: 0 \le t \le \delta\}$ grows at most like $1/\delta$, and so $\sup\{|x(t)|: 0 \le t \le \delta\}$ is bounded as $\delta \to 0$.
\endproof

Now for $0 < \delta < T/2$ define $\{y_\delta(t): 0 \le t \le \delta\}$ to be a path as in Lemma \ref{lem path} from $(x_0,y_0)$ to $(\hat{x},0)$, and define $\{\tilde{y}_\delta(t): 0 \le t \le \delta\}$ to be a path as in Lemma \ref{lem path} from $(\hat{x},0)$ to $(x_1,y_1)$.  Then define $\{y_\delta(t): 0 \le t \le T\}$ by   
     $$
     y_\delta(t) = \begin{cases} y_\delta(t) & \mbox{ if } 0 \le t < \delta\\
                      0 & \mbox{ if } \delta \le t \le T-\delta \\
                      \tilde{y}_\delta(t-(T-\delta)) & \mbox{ if } T-\delta \le t \le T
                      \end{cases}
     $$
and consider the solution $\{(x_\delta(t),\theta_\delta(t)): 0 \le t \le T\}$ of \eqref{ode} with $y(t)$ replaced by $y_\delta(t)$ and with initial condition $(x_0,\theta_0)$.  It is clear that $x_\delta(T) = x_1$.  It remains to show that $|\theta_\delta(T) - \theta_1| < \e$.  

Since $x_\delta(t)$ stays bounded as $\delta \to 0$ there exists $K$ such that $\sup_{0 \le t \le T}|g(x_\delta(t),\theta_\delta(t)| \le K$ for all $0 < \delta < T/2$.  We have
      $$
      |\theta_\delta(\delta) - \theta_0| \le K\delta
      $$
and
  $$
  |\theta_\delta(T-\delta) - \theta_1| = |\Phi(T-2\delta,\theta_\delta(\delta) - \Phi(T,\theta_0)| \le K_1(2\delta + |\theta_\delta(\delta) - \theta_0|) \le K_1(2+K)\delta
  $$
and 
 $$
 |\theta_\delta(T) - \theta_1| \le |\theta_\delta(T) - \theta_\delta(T-\delta)| + |\theta_\delta(T-\delta)-\theta_1| 
 \le K\delta+ K_1(2+K)\delta.
 $$
So given $\e > 0$ choose $\delta > 0$ so that $ (K+ K_1(2+K))\delta < \e$,   
and we are done.  \endproof

\bibliographystyle{plain}
\bibliography{shear.ref}

\end{document}